\documentclass{amsart}
\usepackage[utf8]{inputenc}
\usepackage{amsmath}
\usepackage{amssymb}
\usepackage{mathrsfs}
\usepackage{graphicx}

\usepackage[colorlinks]{hyperref}
\usepackage[svgnames, dvipsnames, usenames]{xcolor}
\hypersetup{urlcolor = RedViolet, linkcolor = RoyalBlue, citecolor = ForestGreen}

\usepackage{tikz}
\usetikzlibrary{intersections}

\usepackage{strdiag}

\def\ppen{\penalty 300 }
\let\col=\colon
\def\colon{\col\ppen}

\theoremstyle{plain} 

\newtheorem{thm}{Theorem}[section]
\newtheorem{prop}[thm]{Proposition}
\newtheorem{lem}[thm]{Lemma}
\newtheorem{cor}[thm]{Corollary}

\theoremstyle{definition}
\newtheorem{defn}[thm]{Definition}
\newtheorem{rem}[thm]{Remark}
\newtheorem{ex}[thm]{Example}

\newtheorem{nota}[thm]{Notation}

\numberwithin{equation}{section}

\renewcommand{\theta}{\vartheta}
\renewcommand{\phi}{\varphi}
\renewcommand{\epsilon}{\varepsilon}
\renewcommand{\subset}{\subseteq}

\newcommand{\Ass}{\mathbf{A}}
\newcommand{\Bss}{\mathbf{B}}
\newcommand{\Alg}{\mathscr{A}}
\newcommand{\Blg}{\mathscr{B}}
\newcommand{\C}{\mathbb{C}}
\newcommand{\Lin}{\mathscr{L}}
\newcommand{\id}{\mathrm{id}}
\newcommand{\N}{\mathbb{N}}

\newcommand{\Z}{\mathbb{Z}}
\newcommand{\Ctrans}{^\mathsf{T}}
\newcommand{\CTr}{\mathop{\mathsf{Tr}}}
\DeclareMathOperator{\spanlin}{span}
\DeclareMathOperator{\Tr}{Tr}
\DeclareMathOperator{\Irr}{Irr}

\begin{document}
\title{Quantum association schemes}
\author{Daniel Gromada}
\address{Czech Technical University in Prague, Faculty of Electrical Engineering, Department of Mathematics, Technická 2, 166 27 Praha 6, Czechia}
\email{gromadan@fel.cvut.cz}
\thanks{I would like to thank Chris Godsil for answering my question on MathOverflow, by which he pointed my attention to the duality of association schemes. I would like to thank Junichiro Matsuda for discussions about quantum Cayley graphs.}
\thanks{I thank to Peter Zeman and Paweł Kasprzak for pointing out certain mistakes in an earlier version of this manuscript.}
\date{\today}
\subjclass{05C25, 05E30, 16T30, 18M40, 20G42}
\keywords{Hopf algebra, quantum group, quantum graph, association scheme, duality, Schur ring}

\begin{abstract}
We introduce quantum association schemes. This allows to define distance regular and strongly regular quantum graphs. We bring examples thereof. In addition, we formulate the duality for translation quantum association schemes corresponding to finite quantum groups.
\end{abstract}

\maketitle
\section*{Introduction}
An \emph{association scheme} over a finite set $X$ is a partition $R$ of $X\times X$ satisfying certain properties. It can also be viewed as a set of directed graphs $\{G_0,\dots,G_d\}$ over the vertex set $X$ such that their sets of edges partition $X\times X$. A third viewpoint, which we will consider throughout this paper, is obtained by replacing the graphs with their adjacency matrices. That is, an association scheme over $X=\{1,\dots,n\}$ is a set $\Ass=\{A_0,\dots,A_d\}$ of $n\times n$ matrices with entries in $\{0,1\}$ satisfying certain axioms including $\sum_{i=0}^d A_i=J$, where $J$ is the all-one-matrix.

We give a precise definition as Def.~\ref{D.Ass}. Nevertheless, these certain axioms can be equivalently summarized by the condition that the matrices span a \emph{coherent algebra}. That is, $\Alg:=\spanlin\{A_0,\dots,A_d\}$ should be closed under the composition of matrices, the conjugate transposition (denoted by $\dag$), but also under the \emph{Schur product} (multiplying matrices entrywise, here denoted by $\bullet$) and the complex conjugation (denoted by $*$; but this actually holds automatically as matrices with entries in $\{0,1\}$ are real). This coherent algebra is often referred to as the \emph{Bose--Mesner algebra} of the association scheme. Note that the classical Schur product is automatically commutative since the multiplication of complex numbers is commutative. On the other hand, the composition of matrices need not be. Nevertheless, the literature mostly focuses on \emph{commutative} association schemes, where $\Alg$ actually is commutative. In particular, it holds that if all the matrices $A_0,\dots,A_d$ are self-adjoint (i.e.\ if the graphs are undirected), then the association scheme is commutative. It is also worth mentioning at this point that the original set of matrices $\{A_0,\dots,A_d\}$ can be recovered from $\Alg$ as the set of minimal projections with respect to the Schur product.

The purpose of this article is to \emph{quantize} the definition of association schemes in a similar manner as quantum groups generalize groups. The definition of a \emph{quantum graph} already appeared in the literature recently \cite{DSW13,Wea12,Wea21,MRV18}. While the original motivation came from quantum information theory, it raised a lot of interest also in the mathematical community of compact quantum groups. A crucial part of the definition of a quantum graph is the notion of a quantum space. A classical finite set $X=\{1,\dots,n\}$ can be described by an $n$-dimensional commutative algebra (the algebra of all functions on $X$), which contains a distinguished basis of $n$ projections, which can be identified with the elements of the original set $X$. A finite \emph{quantum} space $X$ is then described by a finite-dimensional \emph{non-commutative} C*-algebra $C(X)$. Here, $C(X)$ does not have any preferred basis, so the finite quantum space does not have any actual elements; $X$ is then just a formal symbol defined via the algebra $C(X)$. A quantum graph on a quantum set $X$ is then defined by an adjacency matrix $A\colon l^2(X)\to l^2(X)$ satisfying certain properties. The structure of a finite quantum space then allows to define the Schur product of such maps which now may not be commutative anymore.

Now, one idea to quantize the definition of an association scheme might be to replace the finite classical set $X$ by a finite quantum space and replace the adjacency matrices $A_0,\dots,A_d$ by some quantum ones. This definition has a possibly disappointing consequence, namely that the associated coherent algebra $\spanlin\{A_0,\dots,A_d\}$ will again be commutative with respect to the Schur product. We are going to call them \emph{cocommutative quantum association schemes} and study them in Section \ref{sec.ccqas}. We are going to show that there is no need for disappointment as even such a setting provides us interesting results and examples. Indeed, even though the coherent algebra is commutative with respect to the Schur product, all the graphs involved are proper quantum graphs as they are defined over a non-commutative quantum space. Moreover, in Section~\ref{secc.Hadamard}, we provide an example of a cocommutative quantum association scheme based on quantum Hadamard graphs, which is not quantum isomorphic to any classical association scheme. Apart from that, we define and study distance regular and strongly regular quantum graphs in Section~\ref{sec.drg}, which form a particular class of quantum association schemes.

If we require the coherent algebra to be non-commutative with respect to the Schur product, it is not enough to replace the underlying set by a quantum space and the graphs by quantum graphs. We also need to replace the quantum association scheme itself: the \emph{set} of (quantum) graphs has to become a quantum space. That is, we describe the quantum association scheme by an algebra -- the coherent algebra --, which is now non-commutative with respect to the Schur product, so there is no basis of Schur projections and hence there are no distinct elements of the quantum association scheme. 

Note that association schemes are often presented as a generalization of groups. Indeed, for any finite group $\Gamma$, we can define an association scheme over $\Gamma$ by considering the partition $R_g=\{(h,gh)\mid h\in\Gamma\}$. We will call it the \emph{group scheme}. Taking the graph viewpoint, the group scheme is made out of Cayley graphs of $\Gamma$ with respect to sets containing always just one element $\{g\}$. For that reason, quantum coherent algebras can be understood as generalizations of finite quantum groups.

Finally, the main motivation for writing this article and defining quantum association schemes is the concept of duality. One association scheme is said to be dual to another one if swapping the composition with the Schur product (with a proper renormalization) and $\dag$ with $*$ provides an isomorphism of their coherent algebras. This can be generalized to the quantum case in a straightforward way. However, this definition of duality is rather formal and non-constructive. Actually an association scheme need not have any dual\footnote{The Petersen graph is strongly regular with parameters $(10,3,0,1)$. Computing its eigenvalues and formally computing the parameters of the potential dual, we find out that they are not integers and hence the dual does not exist.} and it can also have more than one dual\footnote{For instance, the Shrikhande graph and the $4\times 4$ rook's graph are both strongly regular with parameters $(16,6,2,2)$. They are both self-dual, but since they have the same parameters, they have isomorphic coherent algebra, so they are also dual to each other.}.

Nevertheless, there is a canonical way of how to construct the dual in the special case of abelian groups and, more generally, translation association schemes provided by the Fourier transform. In case of groups, this is also known as Pontryagin duality and it serves as one of the main motivations for introducing quantum groups. The point is that Pontryagin duality (as well as the duality for translation association schemes) is defined only for abelian groups. This is because the duality essentially exchanges the multiplication and the comultiplication. Hence, the dual of a non-abelian group is a quantum group. So, we may expect that the dual of a non-commutative translation association scheme should be a quantum association scheme. And this is exactly what we show in Section~\ref{secc.trans}.

The duality result for commutative translation association schemes was formulated in \cite{Tam63,Del73}. Some attempt to formulate the duality also for non-commutative association schemes was done in \cite{Ban82}, but it is somewhat weaker as it does not give any association scheme structure to the dual (well, as we said, the correct structure is the structure of a quantum association scheme). We comment on this result more in detail in Remark~\ref{R.Bannai}.

Finally, we conclude the article by introducing quantum Latin squares, which provide an additional way to construct strongly regular quantum graphs.

\section{Finite quantum spaces}

All algebras in this article are unital and defined over $\C$.

\begin{defn}
A \emph{Frobenius $*$-algebra} is a finite-dimensional $*$-algebra $\Alg$ equipped with a positive linear functional $\psi$ such that the bilinear form $\beta\colon(a,b)\mapsto\psi(ab)$ is non-degenerate.
\end{defn}

Any Frobenius $*$-algebra is equipped with an inner product $\langle a,b\rangle=\psi(a^*b)$. Since it acts on itself by left multiplication, it must actually be a C*-algebra. We will denote by $\dag$ the adjoint of any map $T\colon \Alg^{\otimes k}\to \Alg^{\otimes l}$ with respect to this inner product. Note that $\psi=\eta^\dag$ in that case, where $\eta\colon\C\to \Alg$ is the inclusion of the unit $1\mapsto 1_\Alg$. In the following text, we will also denote by $m\colon\Alg\otimes\Alg\to\Alg$ the multiplication map $m(a\otimes b)=ab$.

\begin{defn}
A Frobenius $*$-algebra is called
\begin{itemize}
\item \emph{special} if $mm^\dag=\id$,
\item \emph{symmetric} if $\psi$ is tracial i.e.\ if the bilinear form is symmetric.
\end{itemize}
\end{defn}

\begin{ex}
With any finite set $X$, we associate a commutative (and hence symmetric) special Frobenius $*$-algebra $C(X)$ of functions over $X$. The functional is given by summation $\psi(f)=\sum_{x\in X}f(x)$ and hence the inner product is the standard $l^2$-inner product $\langle f,g\rangle=\sum_{x\in X}\overline{f(x)}g(x)$.
\end{ex}

\begin{ex}
The algebra of all $n\times n$ matrices $M_n(\C)$ is a special symmetric Frobenius $*$-algebra with $x^*$ being the conjugate transpose of $x$ and $\psi(x)=n\Tr x$ for $x\in M_n(\C)$. We will denote the associated finite quantum space (see definition below) by $X=M_n$.
\end{ex}

\begin{defn}
A \emph{finite quantum space} is a (possibly non-commutative) special symmetric%
\footnote{Some authors do not require the symmetric condition. In that case, it may be convenient to use a slightly altered bilinear form (and inner product), which is symmetric even when choosing a non-tracial $\psi$ \cite{Was23}. See also \cite{Mat22,Mat23} for graphical calculus in the non-symmetric case.}\,%
\footnote{It is also worth mentioning that the structure of a special symmetric Frobenius $*$-algebra on a given C*-algebra always exists and is given uniquely \cite{Ban02,GroQHad}. In this sense, we could have just said that a finite quantum space is a finite-dimensional C*-algebra and define the bilinear form and the inner product in this unique way afterwards.}
Frobenius $*$-algebra. We use a specific notation and terminology here. We will usually denote the finite quantum space by a letter $X$ and treat it as some abstract object. Then the actual Frobenius $*$-algebra will be denoted by $C(X)$ as if it were functions over some set $X$. In addition, we denote by $l^2(X)$ the associated Hilbert space ($l^2(X)\simeq C(X)$ as a vector space with inner product $\langle a,b\rangle=\psi(a^*b)$).
\end{defn}

The notion of a finite quantum space (some authors say \emph{finite quantum set} or a \emph{quantum algebra}) started to be heavily used recently in the context of \emph{quantum graphs} \cite{MRV18} (which we will mention here later as well; see also \cite{GQGraph,GroQHad}). Nevertheless, the idea is certainly older (see also \cite{Vic11,Ban02}): The point is that classical and quantum physics (in particular, classical and quantum information theory) essentially differs by non-commutativity, by replacing the classical (often finite) set of points/states/outcomes with a matrix algebra (or, more generally, a C*-algebra).

We will often work with some linear maps $T\colon\Alg^{\otimes k}\to\Alg^{\otimes l}$. Formulas involving such linear maps, their compositions and tensor products are often hard to read when written the classical way. They are much easier to understand if they are written using string diagrams. For that purpose, we will denote
\begin{align*}
\spider{2/1}&:=m\colon l^2(X)\otimes l^2(X)\to l^2(X)&&\qquad\text{the multiplication on $C(X)$}\\
\spider{0/1}&:=\eta\in l^2(X)&&\qquad\text{the unit of $C(X)$},\\
\spider{1/0}&:=\psi=\eta^\dag\colon l^2(X)\to\C&&\qquad\text{the corresponding functional},\\
\spider{2/0}&:=\beta=\psi\circ m=
\Diagram{\Dmor{bcirc}2/1 (1,0.5) \Dmor{bcirc}0/0 (1,1)}
\colon l^2(X)\otimes l^2(X)\to\C
&&\qquad\text{the bilinear form.}
\end{align*}

Now, for instance, the associativity of the multiplication $m(\id\otimes m)=m(m\otimes\id)$ can be written as
$\Diagram{\Dmor{bcirc}2/1 (1,1) \Dmor{bcirc}2/1 (1.5,0) \draw (0.5,0.5) -- (0.5,-0.5);}=
 \Diagram{\Dmor{bcirc}2/1 (2,1) \Dmor{bcirc}2/1 (1.5,0) \draw (2.5,0.5) -- (2.5,-0.5);}$.
The fact that $\eta$ is a unit $m(\id\otimes\eta)=\id=m(\eta\otimes\id)$ is written as 
$\Diagram{\Dmor{bcirc}2/1 (1,.5) \Dmor{bcirc}0/0 (1.5,0)}=\Did=
 \Diagram{\Dmor{bcirc}2/1 (1,.5) \Dmor{bcirc}0/0 (0.5,0)}$.

These maps actually generate a monoidal $\dag$-category that admits a nice diagrammatic calculus, which makes diagrammatic computing with these maps very simple. This category or the diagrammatic calculus is known under different names in different contexts: Some people just talk about Frobenius algebras or Frobenius monoids \cite{Vic11}, some people about \emph{spiders} \cite{CK17}, the category is known also under the name \emph{topological quantum field theory} \cite{Koc03} or \emph{category of partitions} \cite{Jon94,BS09}. We will not explain the rules in full detail, but only mention a couple of relations that we are going to need.

For instance, it always holds that $(\beta\otimes\id)(\id\otimes\beta^\dag)=\id=(\id\otimes\beta)(\beta^\dag\otimes\id)$. This is known as the \emph{snake equation} as it can be conveniently expressed using diagrams if we define $\spider{0/2}:=\beta^\dag=\spider{2/0}^\dag$
\begin{equation}\label{eq.snake}
\Diagram{\Dmor{bcirc}2/0 (1,1) \Dmor{bcirc}0/2 (2,0) \draw (0.5,0.5) -- (0.5,-0.5); \draw (2.5,0.5) -- (2.5,1.5);}=
\Did=
\Diagram{\Dmor{bcirc}2/0 (2,1) \Dmor{bcirc}0/2 (1,0) \draw (2.5,0.5) -- (2.5,-0.5); \draw (0.5,0.5) -- (0.5,1.5);}
\end{equation}
Likewise, if we define $\spider{1/2}:=m^\dag$, we can express the \emph{Frobenius law}:
\begin{equation}\label{eq.Flaw}
\Diagram{\Dmor{bcirc}2/1 (1,1) \Dmor{bcirc}1/2 (2,0) \draw (0.5,0.5) -- (0.5,-0.5); \draw (2.5,0.5) -- (2.5,1.5);}=
\Diagram{\Dmor{bcirc}2/1 (1,0) \Dmor{bcirc}1/2 (1,1)}=
\Diagram{\Dmor{bcirc}2/1 (2,1) \Dmor{bcirc}1/2 (1,0) \draw (2.5,0.5) -- (2.5,-0.5); \draw (0.5,0.5) -- (0.5,1.5);}
\end{equation}
Finally, it always holds that $\eta^\dag\eta=\beta\beta^\dag=\dim\Alg$ or, diagrammatically,
$$
\Diagram{\Dmor{bcirc} 0/1 (1,0) \Dmor{bcirc} 1/0 (1,1)}=
\Diagram{\Dmor{bcirc} 0/2 (1,0) \Dmor{bcirc} 2/0 (1,1)}=
\dim\Alg.$$

For more details and examples on the diagrammatic calculus, see e.g. \cite{Vic11,MRV18,GQGraph,GroQHad}.

The bilinear form $\beta=\spider{2/0}$ on $l^2(X)$ induces a bilinear form on $l^2(X)^{\otimes k}$ by $\beta_k=
\Diagram{\Dmor{bcirc}[-000000-/] (1,1.5) \Dmor{bcirc}[-0000-/] (1,1) \Dmor {bcirc}[.--./] (1,0.5) \draw (-2.5,0) -- (-2.5,1); \draw (-1.5,0) -- (-1.5, 0.5); \draw (3.5,0) -- (3.5,0.5); \draw (4.5,0) -- (4.5,1);}
=
\beta(\id\otimes\beta\otimes\id)\cdots(\id_{k-1}\otimes \beta\otimes\id_{k-1})\colon l^2(X)^{\otimes k}\otimes l^2(X)^{\otimes k}\to\C$.
For any $A\colon l^2(X)^{\otimes k}\to l^2(X)^{\otimes l}$, $k,l\in\N_0$, we define the \emph{categorical transpose} using this bilinear form.
$$A\Ctrans=(\beta_l\otimes\id_k)(\id_l\otimes A\otimes\id_k)(\id_l\otimes\beta_k)=
\Diagram{\DMor{square}[-.-/-.-] (5,0.5) {$\scriptstyle A$}
         \Dmor{bcirc}2/0 (3.5,2) \DMor{bcirc}[-0000-/] (3.5,2.5){}
         \Dmor{bcirc}0/2 (6.5,-1) \DMor{bcirc}[/-0000-] (6.5,-1.5){}
         \draw (1,1.5) -- (1,-1.5); \draw (3,1.5) -- (3,-1.5);
         \draw (7,-0.5) -- (7,2.5); \draw (9,-0.5) -- (9,2.5);}
$$
In particular, for any $f\in l^2(X)$, we define $f\Ctrans=\beta(\id\otimes f)\colon l^2(X)\to\C$. Note also that (thanks to the bilinear form being symmetric) the categorical transpose is involutive, so $(A\Ctrans)\Ctrans=A$ and it is a contravariant functor, so $(AB)\Ctrans=B\Ctrans A\Ctrans$.

For any $A\colon l^2(X)^{\otimes k}\to l^2(X)^{\otimes l}$, we define $A^*=(A^\dag)\Ctrans$. Note first that on $l^2(X)$, this map coincides with the $*$ involution we already have here since
$$\beta((f^\dag)\Ctrans\otimes g)=
\Diagram{\Dmor{bcirc}0/2 (1,0) \Dmor{bcirc}2/0 (2,1)
         \DMor{covec}1/0 (0.5,1.5) {$\scriptstyle f^\dag$}
         \DMor{vec}0/1 (2.5,-.5) {$\scriptstyle g$}}
=
\Diagram{\DMor{covec}1/0 (1,1.5) {$\scriptstyle f^\dag$}
         \DMor{vec}0/1 (1,-.5) {$\scriptstyle g$}}
=\langle f,g\rangle=\beta(f^*\otimes g).
$$
Consequently, we have $A^*(f)=(A(f^*))^*$ (which we could have taken as a definition) and hence the $*$ operation actually does not depend on the particular choice of the bilinear form. Also note that $A=A^*$ if and only if $A$ is $*$-preserving and that $*$ is a (covariant) functor, so $(AB)^*=A^*B^*$.

We define the \emph{Schur product} of maps $A,B\colon l^2(X)\to l^2(X)$ as
$$
A\bullet B=m(A\otimes B)m^\dag=
\Diagram{\Dmor{bcirc}2/1 (1,2) \Dmor{bcirc}1/2 (1,-1)
         \DMor{square}1/1 (0.5,0.5) {$\scriptstyle A$}
         \DMor{square}1/1 (1.5,0.5) {$\scriptstyle B$}
}\colon l^2(X)\to l^2(X).
$$

As one can easily check, if $X$ is just a set, this definition corresponds to the classical entrywise Schur product.

Considering this product, we have
\begin{equation}\label{eq.bulletinv}
(A\bullet B)^*=B^*\bullet A^*,\quad (A\bullet B)\Ctrans=B\Ctrans\bullet A\Ctrans,\quad (A\bullet B)^\dag=A^\dag\bullet B^\dag.
\end{equation}
In fact, it can be shown that the algebra of linear maps $A\colon l^2(X)\to l^2(X)$ with respect to the Schur product and the $*$-involution is a C*-algebra isomorphic to $C(X)\otimes C(X)^{\rm op}$ \cite[Proposition~1.17]{GQGraph}. Note that it has the unit $J=\eta\eta^\dag$. Indeed:
$$A\bullet J=
\Diagram{\Dmor{bcirc}2/1 (1,2) \Dmor{bcirc}1/2 (1,-1)
         \DMor{square}1/1 (0.5,0.5) {$\scriptstyle A$}
         \Dmor{bcirc}1/0 (1.5,0) \Dmor{bcirc}0/1 (1.5,1)
}
=
\Diagram{\DMor{square}1/1 (1,0.5) {$\scriptstyle A$}}
$$

Observe that the special condition for the Frobenius algebra equivalently says that $I\bullet I=I$. Since clearly $I^*=I$, it means that $I$ is an orthogonal projection with respect to the Schur product and $*$-involution.

Finally, for any $A\colon l^2(X)\to l^2(X)$, we define the \emph{categorical trace} as
$$
\CTr A=
\Diagram{\Dmor{bcirc}2/0 (1,2) \Dmor{bcirc}0/2 (1,-1)
         \DMor{square}1/1 (0.5,0.5) {$\scriptstyle A$} \draw (1.5,-0.5) -- (1.5,1.5);
}
$$

Again, thanks to the bilinear form being symmetric, the trace is indeed tracial as we have
$$
\Diagram{\Dmor{bcirc}2/0 (1,3) \Dmor{bcirc}0/2 (1,-2)
         \DMor{square}1/1 (0.5,1.5) {$\scriptstyle A$}
         \DMor{square}1/1 (0.5,-.5) {$\scriptstyle B$}
         \draw (1.5,-1.5) -- (1.5,2.5);}=
\Diagram{\Dmor{bcirc}[-0-/] (1,2) \Dmor{bcirc}[/-0-] (1,-1)
         \DMor{square}1/1 (0,0.5) {$\scriptstyle A$}
         \DMor{square}1/1 (2,0.5) {$\scriptstyle B\Ctrans$}}=
\Diagram{\Dmor{bcirc}2/0 (1,3) \Dmor{bcirc}0/2 (1,-2)
         \DMor{square}1/1 (0.5,1.5) {$\scriptstyle B$}
         \DMor{square}1/1 (0.5,-.5) {$\scriptstyle A$}
         \draw (1.5,-1.5) -- (1.5,2.5);}=
$$
that is,
\begin{equation}\label{eq.TrSchur}
\CTr(AB)=\eta^\dag(A\bullet B\Ctrans)\eta=\CTr(BA).
\end{equation}
Actually, it turns out that, in the symmetric case, the categorical trace coincides with the standard notion of a trace of some linear operator.

\section{Quantum association schemes}
\label{sec.ccqas}

\subsection{Coherent algebras}
\label{secc.coherent}
As we mentioned in the introduction, the most general way how to quantize association schemes, is just to quantize coherent algebras:

\begin{defn}\label{D.Qcoherent}
Let $X$ be a finite quantum space. Denote by $\Lin(l^2(X))$ the vector space of linear operators on $l^2(X)$. A \emph{quantum coherent algebra} over $X$ is a vector subspace $\Alg\subset\Lin(l^2(X))$ that is closed under both the composition as well as the Schur product, both the $\dag$ involution as well as the $*$ involution and contains both the identity $I$ as well as the element $J:=\eta\eta^\dag$. That is, it is a unital involutive algebra with respect to both the products and involutions. We say that it is \emph{commutative} if it it is commutative with respect to the composition and \emph{cocommutative} if it is commutative with respect to the Schur product. We will say that the quantum coherent algebra is \emph{homogeneous} if $I$ is a minimal projection with respect to the Schur product. This means that, if $P\bullet P=P^*=P=P\bullet I=I\bullet P$, then $P=I$ or $P=0$.
\end{defn}


Since we want to study the duality, let us define what we mean by that right away:

\begin{defn}\label{D.dual}
Suppose $\Alg$ and $\Blg$ are quantum coherent algebras over quantum spaces $X$ and $Y$, both with $n:=\dim C(X)=\dim C(Y)$. Then we say that $\Blg$ is dual to $\Alg$ if there is a vector space isomorphism $\Phi\colon\Alg\to\Blg$ such that, for every $x,y\in\Alg$,
\begin{align*}
    \Phi(xy)&=\Phi(x)\bullet\Phi(y)    &n\Phi(x\bullet y)&=\Phi(x)\Phi(y)\\
\Phi(x^\dag)&=\Phi(x)^*                &        \Phi(x^*)&=\Phi(x)^\dag.
\end{align*}
\end{defn}

\begin{rem}
Since $I$ is the identity with respect to the composition and $J=\eta\eta^\dag$ is the identity with respect to the Schur product, it is clear that such a map will will always assign $I\mapsto J$ and $J\mapsto nI$.

Note also that the factor $n$ in the second equation is necessary. Without it, we would have to map $J\mapsto I$ and $I\mapsto J$, which would, however, not be compatible with the fact that
$$I\bullet I=I,\qquad JJ=nJ.$$
\end{rem}

\subsection{Cocommutative association schemes}
\label{secc.CQAS}
First, recall what a classical association scheme exactly is. For an introduction to the classical theory, see also \cite{BI84,BCN89,MT09,God10}.

\begin{defn}
\label{D.Ass}
Consider the set $X=\{1,\dots,n\}$ with $n\in\N$. A $d$-class \emph{association scheme} over $X$ is a set $\Ass=\{A_0,A_1,\dots,A_d\}$ of $n\times n$ matrices $A_i$ such that
\begin{enumerate}
\item All $A_i$ have entries in $\{0,1\}$,
\item $A_0=I$,
\item $\sum_{i=0}^d A_i=J$, where $J$ is the all ones matrix
\item $A_i^{\rm T}\in\Ass$ for every $i$,
\item $A_iA_j\in\spanlin\Ass$ for every $i,j$,
\end{enumerate}
We say that $\Ass$ is
\begin{itemize}
\item \emph{commutative} if $A_iA_j=A_jA_i$ for every $i,j$,
\item \emph{symmetric} if $A_i=A_i^{\rm T}$ for every $i$.
\end{itemize}
\end{defn}

In the following definition, we quantize the concept by replacing the classical adjacency matrices $A_i$ over the classical space $X$ by their quantum counterparts.

\begin{defn}\label{D.QAss}
Let $X$ be a finite quantum space. A $d$-class \emph{cocommutative quantum association scheme} (CQAS) is a set $\Ass=\{A_0,A_1,\dots,A_d\}$ with $A_i\colon l^2(X)\to l^2(X)$ such that
\begin{enumerate}
\item $A_i=A_i^*$, $A_i\bullet A_i=A_i$ for every $i$,
\item $A_0=I$,
\item $\sum_{i=0}^d A_i=J$,
\item $A_i^\dag\in\Ass$ for every $i$,
\item $A_iA_j\in\spanlin\Ass$ for every $i,j$,
\end{enumerate}
We say that $\Ass$ is
\begin{itemize}
\item \emph{commutative} if $A_iA_j=A_jA_i$ for every $i,j$,
\item \emph{symmetric} if $A_i=A_i^\dag$ for every $i$.
\end{itemize}
\end{defn}

As a first observation, being symmetric implies commutativity as for the classical schemes:

\begin{prop}
A symmetric CQAS is always commutative.
\end{prop}
\begin{proof}
Since $A_i=A_i^\dag=A_i\Ctrans$ for every $i$, we must have $A=A\Ctrans$ for every $A\in\spanlin\Ass$. But then using axiom (5), $A_iA_j=(A_iA_j)\Ctrans=A_j\Ctrans A_i\Ctrans=A_jA_i$.
\end{proof}

The following proposition will explain the adjective \emph{cocommutative}:

\begin{prop}
\label{P.coco}
In any CQAS, we have $A_i\bullet A_j=0=A_j\bullet A_i$ for $i\neq j$.
\end{prop}
\begin{proof}
Recall that operators $l^2(X)\to l^2(X)$ together with the Schur product $\bullet$ and the involution $*$ form a C*-algebra. By axiom (1), the $A_i$ are orthogonal projections with respect to the Schur product and the $*$-involution. By axiom (3), $J=\sum_{i=0}^dA_i$, so they sum up to the (Schur) identity. But any set of orthogonal projection in a C*-algebra that sum to the identity must be mutually orthogonal (and, in particular, commute).
\end{proof}

\begin{prop}
For any CQAS $\Ass=\{A_0,A_1,\dots,A_d\}$, the elements are linearly independent. That is, $\dim\spanlin\Ass=d+1$.
\end{prop}
\begin{proof}
Follows directly from Proposition~\ref{P.coco}.
\end{proof}

Finally, we are ready to explain the relationship with coherent algebras.

\begin{thm}
\label{T.AlgAss}
For any CQAS $\Ass$, the space $\Alg:=\spanlin\Ass$ is a homogeneous cocommutative quantum coherent algebra. Conversely, for any homogeneous cocommutative quantum coherent algebra $\Alg$, there is a unique (up to the order of the matrices $A_1,\dots,A_d$) CQAS such that $\Alg=\spanlin\Ass$.
\end{thm}
\begin{proof}
The first part is clear -- by definition $\Alg$ is closed under both the products and contains both $I$ and $J$.

For the second part, recall again that $\Alg$ is a finite-dimensional commutative C*-algebra with respect to the Schur product and the $*$-involution. Hence, we get the statement just applying the Gelfand duality. The set $\{A_0,A_1,\dots,A_d\}$ is defined to be the set of all minimal projections in $\Alg$ (with $A_0=I$). Checking that such a CQAS satisfies all the required properties is straightforward except for (4).
%

To show (4), note first that since $A_i=A_i^*$, we have $A_i^\dag=A_i\Ctrans$. Now, observe that by Eq.~\eqref{eq.bulletinv}, the categorical transposition is an algebra $*$-isomorphism with respect to the Schur product and $*$-involution. (Actually, in general, it is an antiisomorphism, but thanks to the commutativity, the order of the operation does not matter.) So, if we transpose all elements of $\Alg$ and then construct all the minimal projections, we get the same set. Therefore, $A_i\Ctrans$ must be one of the $A_0,A_1,\dots,A_d$ for every $i$.
\end{proof}

Note that the homogenity assumption on the coherent algebra is important as otherwise the identity $I$ would not be among the minimal projections $A_0,\dots,A_d$ which would contradict axiom (2). This is true also in the classical case of ordinary association schemes. Leaving out this assumption, one would obtain a \emph{coherent configuration} instead of an association scheme. However, many classical publications focus on commutative coherent algebras and commutative association schemes where this assumption holds automatically.

Let us generalize this to the quantum setting. Before we do that, we observe that it is really an issue conserning the ``classical structure'' of our quantum set. In the ``purely quantum'' setting, where $X=M_n$, i.e.\ our finite quantum space is given by the algebra of all matrices $C(X)=M_n(\C)$, this problem disappears:

\begin{prop}
Any coherent algebra over $M_n$ is homogeneous.
\end{prop}
\begin{proof}
By \cite[Remark~3.6]{GQGraph}, $I$ is a minimal projection with respect to the Schur product in the whole algebra of all operators $M_n(\C)\to M_n(\C)$, so it will be a minimal projection in any quanum coherent algebra over $M_n$.
\end{proof}

\begin{lem}
Let $X$ be a finite quantum set. Then $C(X)=\bigoplus_{j=1}^a M_{n_j}(\C)$. Denote $I=\sum_{j=1}^a I_j$ according to this decomposition. Then any Schur projection $P$ which is smaller than $I$ (that is, $P\bullet P=P^*=P=P\bullet I=I\bullet P$) is of the form $P=\sum_{j\in\pi}I_j$ for some $\pi\subseteq\{1,\dots,a\}$.
\end{lem}
\begin{proof}
First, it is well-known that any finite-dimensional C*-algebra is of the given form. Consequently, any map $P\colon l^2(X)\to l^2(X)$ can be expressed as $P=\sum_{i,j=1}^a P_{ij}$, where $P_{ij}$ maps the $j$-th part to the $i$-th part and acts everywhere else trivially. One can check that $P\bullet I=\sum_{j=1}^a P_{jj}\bullet I_j$, so from $P\bullet I=P$, we have $P_{ij}=0$ for $i\neq j$. Also it is easy to see that all $I_j$ are actually Schur projections; in addition, they are minimal by \cite[Remark 3.6]{GQGraph}. From this, the statement follows.
\end{proof}

\begin{prop}
Any commutative quantum coherent algebra is homogeneous.
\end{prop}
\begin{proof}
Suppose that $P$ is a Schur subprojection of $I$ contained in some quantum coherent algebra $\Alg$. Then, according to the previous lemma, we can split $l^2(X)=V_1\oplus V_2$ such that $P$ has the following block structure. Note also the block structure of $J=\eta\eta^\dag\in\Alg$.
$$P=\begin{pmatrix}I_1&0\\0&0\end{pmatrix},\qquad J=\begin{pmatrix}J_{11}&J_{12}\cr J_{21}&J_{22}\end{pmatrix},$$
where $I_1$ is the identity $V_1\to V_1$ and $J_{ij}=\eta_i\eta_j^\dag$ with $\eta_i$ being the unit vector in the $i$-th part. But this means that unless $V_1$ or $V_2$ are trivial, we have $PJ\neq JP$.
\end{proof}

\begin{cor}
For any commutative CQAS $\Ass$, the space $\Alg:=\spanlin\Ass$ is a commutative cocommutative quantum coherent algebra. Conversely, for any commutative cocommutative quantum coherent algebra $\Alg$, there is a unique (up to the order of the matrices $A_1,\dots,A_d$) commutative CQAS such that $\Alg=\spanlin\Ass$.
\end{cor}

\subsection{Duality in the commutative cocommutative setting}
\label{secc.CCQAS}
In this section, we will consider commutative cocommutative quantum association schemes only and shorten them as CCQAS. We will reformulate the classical theory of parameters for association schemes that are used to formulate the formal duality principle.

\begin{thm}
\label{T.AssDual}
Consider a CCQAS $\Ass$ and its Bose--Mesner algebra $\Alg=\spanlin\Ass$. Then there exists a basis $\{E_0,E_1,\dots,E_d\}$ of $\Alg$ such that
\begin{enumerate}
\item $E_i=E_i^\dag$, $E_iE_j=\delta_{ij}E_i$ for every $i$ (i.e.\ they are mutually orthogonal projections),
\item $E_0=\frac{1}{n}J$, $n=\eta^\dag\eta$,
\item $\sum_{i=0}^d E_i=I$,
\item $E_i^*\in\{E_0,E_1,\dots,E_d\}$ for every $i$,
\item $E_i\bullet E_j\in\spanlin\{E_0,E_1,\dots,E_d\}$.
\end{enumerate}
Such a basis is unique up to the order.
\end{thm}
\begin{proof}
We get the theorem repeating the construction from the proof of Thm.~\ref{T.AlgAss}.
\end{proof}

For any CCQAS $\Ass$, we will refer to such a basis $\{E_0,E_1,\dots,E_d\}$ as the \emph{dual basis}. Fixing the order of the dual basis, we will denote by $P$ and $Q:=nP^{-1}$ the transition matrices between the two bases. That is,
\begin{align*}
A_i&=\sum_{r=0}^d P_{ri}E_r,\\
E_j&=\frac{1}{n}\sum_{s=0}^d Q_{sj}A_s
\end{align*}

We call the numbers $P_{ij}$ the \emph{eigenvalues} of $\Ass$ and $Q_{ij}$ the \emph{dual eigenvalues}. The reason for this name is that since all the $A_i$ mutually commute, they can be mutually diagonalized and the $E_r$ are projections on the common eigenspaces corresponding to the eigenvalues $P_{ri}$. Indeed:
\begin{equation}\label{eq.eigenvalues}
A_iE_r=\sum_{s=0}^dP_{si}E_sE_r=P_{ri}E_r.
\end{equation}

Denote $n_i:=P_{0i}$. Then $A_iJ=JA_i=n_iJ$. This equivalently means that 
\begin{equation}\label{eq.AssReg}
A_i\eta=n_i\eta\quad\text{and}\quad\eta^\dag A_i=n_i\eta^\dag.
\end{equation}
Hence, all matrices $A_i$ actually define $n_i$-regular quantum graphs (see Def.~\ref{D.regular}). For that reason, the numbers are called the \emph{valencies} of the scheme.

Dually, we denote by $m_l:=\CTr E_l$ the rank of the projection $E_l$, i.e.\ the dimension of the $l$-th eigenspace. These numbers are called the \emph{multiplicities} of the scheme.

Since $A_iA_j\in\spanlin\Ass$, there must also be numbers $p_{ij}^k$ such that
$$A_iA_j=\sum_{k=0}^dp_{ij}^kA_k,$$
which are called the \emph{intersection numbers} of $\Ass$.

Dually, we also have numbers $q_{ij}^k$ such that
$$E_i\bullet E_j=\frac{1}{n}\sum_{k=0}^d q_{ij}^k E_k,$$
which are called the \emph{Krein parameters}.

\begin{lem}
\label{L.EigenInt}
Let $\Ass$ be a CCQAS and denote its parameters as above. Then
$$p_{ij}^k=\frac{1}{nn_k}\sum_l m_lP_{li}P_{lj}\bar P_{lk}.$$
\end{lem}
\begin{proof}
We start with $A_iA_j=\sum_k p_{ij}^kA_k$. Now, Schur-multiply by $A_k$ from left for some particular $k$. Since $A_i$'s are Schur idempotents, we have $A_k\bullet(A_iA_j)=p_{ij}^k A_k$. Now, sandwich this with $\eta^\dag$ and $\eta$ and use Eq.~\eqref{eq.AssReg} above and Eq.~\eqref{eq.TrSchur} saying that $\eta^\dag(A\bullet B)\eta=\CTr(A\Ctrans B)$ to obtain
$$nn_k p_{ij}^k=p_{ij}^k\eta^\dag A_k\eta=\eta^\dag(A_k\bullet(A_iA_j))\eta=\CTr(A_k\Ctrans A_iA_j)=\CTr(A_k^\dag A_iA_j).$$
Finally, applying $\dag$ to Eq.~\eqref{eq.eigenvalues}, we get $A_i^\dag E_l=\bar P_{li}E_l$. Note also that $I=\sum_l E_l$ and recall the notation $m_l=\CTr E_l$, so
$$nn_k p_{ij}^k=\CTr(A_k^\dag A_iA_j)=\sum_l\CTr(A_k^\dag A_iA_jE_l)=\sum_l m_l\bar P_{lk}P_{li}P_{lj}.\qedhere$$
\end{proof}

\begin{lem}
\label{L.EigenKrein}
Let $\Ass$ be a CCQAS and denote its parameters as above. Then
$$q_{ij}^k=\frac{1}{nm_k}\sum_l n_lQ_{li}Q_{lj}\bar Q_{lk}.$$
\end{lem}
\begin{proof}
The proof is analogous to the preceding lemma, so we will describe it a bit more briefly: Start with the definition of $q_{ij}^k$ and derive
\begin{align*}
\frac{1}{n}q_{ij}^km_k&=\CTr(E_k(E_i\bullet E_j))=\eta^\dag(E_k\Ctrans\bullet E_i\bullet E_j)\eta=\sum_l\eta^\dag(E_k^*\bullet E_i\bullet E_j\bullet A_l)\\&=\frac{1}{n^3}\sum_l \bar Q_{lk}Q_{li}Q_{lj}nn_l.\qedhere
\end{align*}
\end{proof}

\begin{thm}
Suppose $\Ass$ and $\Bss$ are CCQAS and $\Alg$, $\Blg$ are their Bose--Mesner algebras. Then the following are equivalent
\begin{enumerate}
\item $\Blg$ is dual to $\Alg$ (according to Def.~\ref{D.dual}).
\item Eigenvalues of $\Ass$ equal to the dual eigenvalues of $\Bss$.
\item Intersection numbers of $\Ass$ equal to the Krein parameters of $\Bss$.
\end{enumerate}
\end{thm}
Let us briefly explain what exactly we mean by conditions (2) and (3): We consider them as properties of the CCQAS. A CCQAS is defined as a \emph{set} of some operators, so their order is not fixed. Also the dual basis is unique only up to the order. So, by (2) and (3) we mean that there exists an ordering of $\Ass$ and $\Bss$ and there exists an ordering of their dual bases such that (2) the eigenvalue matrix of $\Ass$ equals to the dual eigenvalue matrix of $\Bss$, resp.\ (3) the ``array'' of intersection numbers of $\Ass$ equal to the Krein parameters of $\Bss$.
\begin{proof}
The implication $(3)\Rightarrow(1)$ basically follows from definition: The coherent algebras $\Alg$ and $\Blg$, as abstract algebras, are completely determined by their structure constants, which are the intersection numbers and the Krein parameters. Swapping these two, we swap the two products. Indeed, denote $\Ass=\{A_i\}_{i=0}^d$, $\Bss=\{B_i\}_{i=0}^d$ and denote by $\{E_j\}_{j=0}^d$, $\{F_j\}_{j=0}^d$ their dual bases. Define $\Phi\colon\Alg\to\Blg$ to be the linear extension of $A_i\mapsto nF_i$. Then also
$$\Phi(E_j)=\frac{1}{n}\sum_{s=0}^dQ_{sj}\Phi(A_s)=\sum_{s=0}^dQ_{sj}F_s=B_j$$
Now, taking any two elements of $\Alg$: $A=\sum_i\alpha_iA_i=\sum_j\tilde\alpha_jE_j$, $B=\sum_i\beta_iA_i=\sum_j\tilde\beta_jE_j$, we have
\begin{align*}
\Phi(AB)&=\sum_{ij}\alpha_i\beta_j\Phi(A_iA_j)=\sum_{ijk}\alpha_i\beta_jp_{ij}^k\Phi(A_k)=n\sum_{ijk}\alpha_i\beta_jp_{ij}^kF_k\\&=n^2\sum_{i,j}\alpha_i\beta_j\,F_i\bullet F_j=\sum_i\alpha_i\Phi(A_i)\sum_j\beta_j\Phi(A_j)=\Phi(A)\bullet\Phi(B),\\
n\Phi(A\bullet B)&=n\sum_{ij}\tilde\alpha_i\tilde\beta_j\Phi(E_i\bullet E_j)=\sum_{ij,k}\tilde\alpha_i\tilde\beta_jq_{ij}^k\Phi(E_k)=\sum_{i,j,k}\tilde\alpha_i\tilde\beta_jq_{ij}^kB_k\\&=\sum_{i,j}\tilde\alpha_i\tilde\beta_jB_iB_j=\sum_i\tilde\alpha_i\Phi(E_i)\sum_j\tilde\alpha_j\Phi(E_j)=\Phi(A)\Phi(B)
\end{align*}
The fact that $\Phi$ preserves the involutions in the prescribed way follows from the fact that the bases $\{A_i\}$ and $\{B_i\}$ are self-conjugated and $\{E_j\}$, $\{F_j\}$ are self-adjoint. That is, $\Phi(A^*)=\sum_i\bar\alpha_i\Phi(A_i)=\sum_i\bar\alpha_iE_i=\Phi(A)^\dag$ and analogous for $\Phi(A^\dag)=\Phi(A)^*$.

The implication $(1)\Rightarrow(2)$ follows from the uniqueness in Theorems~\ref{T.AlgAss}, \ref{T.AssDual}. Denote $\Ass=\{A_i\}_{i=0}^d$ and by $\{E_j\}_{j=0}^d$ its dual basis. Now, it is enough to check that $B_i:=\Phi(E_i)$, $F_j=\Phi(A_j)$ satisfy the conditions in Definition~\ref{D.QAss} and Theorem~\ref{T.AssDual}. But this is obvious.

Finally, implication $(2)\Rightarrow(3)$ follows from Lemmata~\ref{L.EigenInt}, \ref{L.EigenKrein} as intersection numbers and Krein parameters are determined by the eigenvalues and dual eigenvalues. So, if the latter coincide, the former must as well.
\end{proof}

If the above properties are satisfied, we say that $\Bss$ is \emph{dual} to $\Ass$.

\section{Distance regular quantum graphs}
\label{sec.drg}

\subsection{Quantum graphs}

\begin{defn}
Let $X$ be a quantum space. A \emph{(directed) quantum graph} on $X$ is determined by an \emph{adjacency matrix} $A\colon l^2(X)\to l^2(X)$ such that $A\bullet A=A=A^*$. We say that the graph is \emph{undirected} if $A=A^\dag$. We say that it has no loops if $A\bullet I=0=I\bullet A$. An undirected quantum graph with no loops will be called \emph{simple}.
\end{defn}

Observe that a CQAS $\Ass=\{A_0,\dots,A_d\}$ consists of $d$ directed quantum graphs, which together partition the full graph $J-I=\sum_{i=1}^d A_i$. The matrix $A_0=I$ describes a quantum graph that only consists of loops at every vertex and nothing else. The CQAS is symmetric if and only if all the graphs are undirected.

\begin{defn}
For any quantum graph without loops described by some adjacency matrix $A$, we define its \emph{complement} to be the quantum graph defined by $J-I-A$.
\end{defn}

\begin{defn}[{\cite[Def.~2.24]{Mat22}}]
\label{D.regular}
Let $X$ be a quantum space. A quantum graph with adjacency matrix $A\colon l^2(X)\to l^2(X)$ is said to be \emph{regular} of degree $k$ if $A\eta=k\eta$ and $\eta^\dag A=k\eta^\dag$.
\end{defn}

Note that in case of undirected quantum graphs, the two conditions in the definition of a regular graph are actually equivalent.

Finally, Matsuda defines in \cite{Mat23} the notion of connectedness. Since the definition is slightly complicated, we bring the following equivalent characterization from \cite[Theorem~3.7]{Mat23} that works for regular quantum graphs only.

\begin{defn}[\cite{Mat23}]
A $k$-regular quantum graph is said to be \emph{connected} if $k$ is a simple eigenvalue of its adjacency matrix.
\end{defn}

Finally, we would like to remind the definition of a quantum isomorphism for quantum graphs. This has several equivalent characterizations summarized in \cite{BCE+20}. We took the one from \cite[Theorem~4.7]{BCE+20}.

\begin{defn}
\label{D.qiso}
Let $X_1,X_2$ be quantum spaces and denote by $m_1,\eta_1,m_2,\eta_2$ the defining structure maps (multiplication and unit). Suppose $A_i\colon l^2(X_i)\to l^2(X_i)$, $i=1,2$ are adjacency matrices of quantum graphs. We say that the graphs are \emph{quantum isomorphic} if there is an isomorphism of monoidal $\dag$-categories $\langle m_1,\eta_1,A_1\rangle\to\langle m_2,\eta_2,A_2\rangle$ mapping $m_1\mapsto m_2$, $\eta_1\mapsto\eta_2$, $A_1\mapsto A_2$.
\end{defn}

Without discussing any further what exactly the definition means, let us mention the important consequence for this work: The definition says that a quantum isomorphism preserves every property of a quantum graph that can be expressed via the adjacency matrix $A$, the multiplication $m$ and the unit $\eta$. In particular, it preserves the coherent algebra generated by the adjacency matrix.

\subsection{Distance regular quantum graphs}

\begin{defn}
Let $X$ be a quantum space and $A\colon l^2(X)\to l^2(X)$ an adjacency matrix of a quantum graph. We say that the quantum graph is \emph{distance regular} with diameter $d$ if $\Alg=\spanlin\{I,A,A^2,\dots,A^d\}$ is a $(d+1)$-dimensional coherent algebra.
\end{defn}

\begin{prop}
Any distance regular quantum graph is regular.
\end{prop}
\begin{proof}
We assume that the algebra $\Alg$ is a coherent algebra, so it must contain the element $J=\eta\eta^\dag$, which is a multiple of a rank one projection (on the subspace $\spanlin\{\eta\}$). The algebra is obviously commutative, so $AJ=JA$. Together, it follows that $J$ is a projection on some eigenspace of $A$, namely that $\eta$ is an eigenvector of $A$. That is, $A\eta=k\eta$ and $\eta^\dag A=k\eta^\dag$ for some $k\in\C$.
\end{proof}

\begin{prop}\label{P.drgConn}
Any distance regular quantum graph is connected.
\end{prop}
\begin{proof}
The commutative algebra $\Alg=\spanlin\{I,A,\dots,A^d\}$ is spanned by the projection on the eigenspaces of $A$. By assumption of being a coherent algebra, it must contain the projection $J$, which is of rank one and corresponds to the eigenvalue $k$, where $k$ is the degree of $A$.
\end{proof}

Below, we will focus on the cocommutative setting. We leave the question open whether there exists a distance regular quantum graph with coherent algebra that is not cocommutative.

\subsection{Cocommutative distance regular quantum graphs}

\begin{prop}
Let $X$ be a quantum space and $A\colon l^2(X)\to l^2(X)$ an adjacency matrix of a quantum graph with no loops which is distance regular with diameter~$d$. Assume, in addition, that $A^k\bullet A^l=A^l\bullet A^k$ for every $k,l\in\{0,\dots,d\}$. Define the sequence $(A_i)$ by
\begin{equation}\label{eq.DT}
A_0=I,\quad A_1=A,\quad A_i=\frac{1}{\beta_i}B_i,\quad B_i=A^i-\sum_{j=0}^{i-1}A^i\bullet A_j,\quad\beta_i=\frac{\eta^\dag B_i\bullet B_i\eta}{\eta^\dag(B_i)\eta}.
\end{equation}
Then $\Ass=(A_0,\dots,A_d)$ is a CCQAS.
\end{prop}
\begin{proof}
First, recall that by the assumption of distance regularity, the set $\{A^i\}_{i=0}^d$ is linearly independent and spans a coherent algebra, which is clearly commutative with respect to composition. Now, notice that the transformation $(A^i)\to(A_i)$ is triangular with ones on diagonal, so it is invertible. Hence, the set $\{A_i\}_{i=0}^d$ must also be linearly independent, commutative and span the same algebra.

From the assumption $A^k\bullet A^l=A^l\bullet A^k$, it follows that the elements $A_i$ commute also with respect to the Schur product. Now, it remains to show that $\{A_0,\dots,A_d\}$ is the unique basis of the coherent algebra $\Alg=\spanlin\{I,A,\dots,A^d\}$ by mutually orthogonal Schur projections. Indeed, in the proof of Theorem~\ref{T.AlgAss}, we showed that such a basis must already satisfy all the axioms of a cocommutative quantum association scheme.

The fact that $A_i=A_i^*$ follows from $A=A^*$ and the fact that $*$ is functorial.

Now, we prove that $A_i\bullet A_j=0$ for $i\neq j$ by induction. The base case $A_1\bullet A_0=A\bullet I=0$ holds by our assumption that the graph has no loops. Now, suppose that $i>j$ and, as induction hypothesis, assume that $A_k\bullet A_l=0$ whenever $k,l<i$. Then $\beta_iA_i\bullet A_j=A^i\bullet A_j-\sum_{k=0}^{i-1}A^i\bullet A_k\bullet A_j=A^i\bullet A_j-A^i\bullet A_j=0$. Also note that $A_i\bullet A_i\neq 0$ since otherwise we would have $0=\eta^\dag (A_i\bullet A_i)\eta=\CTr(A_i\Ctrans A_i)$ and hence $A_i=0$.

Finally, denote $A_i\bullet A_i=\sum\alpha_j A_j$. But then $0=A_i\bullet A_i\bullet A_k=\alpha_k A_k\bullet A_k$ for $k\neq i$, so $\alpha_k=0$ and $A_i\bullet A_i=\alpha_i A_i$. Finally,
$$\frac{(\eta^\dag B_i\eta)^2}{\eta^\dag(B_i\bullet B_i)\eta}=\eta^\dag(A_i\bullet A_i)\eta=\alpha_i\eta^\dag A_i\eta=\alpha_i\frac{(\eta^\dag B_i\eta)^2}{\eta^\dag(B_i\bullet B_i)\eta},$$
so $\alpha_i=1$.
\end{proof}

\begin{defn}
A distance regular quantum graph with cocommutative coherent algebra will be called \emph{cocommutative}.
\end{defn}

\begin{prop}
Let $X$ be a quantum space and $A\colon l^2(X)\to l^2(X)$ an adjacency matrix of a cocommutative distance regular quantum graph with no loops. Then the corresponding CCQAS can be equivalently defined by
\begin{equation}\label{eq.DT2}
A_i=\frac{1}{\tilde\beta_i}\tilde B_i,\quad \tilde B_i=AA_{i-1}-\sum_{j=0}^{i-1}(AA_{i-1})\bullet A_j,\quad\tilde\beta_i=\frac{\eta^\dag\tilde B_i\bullet\tilde B_i\eta}{\eta^\dag(\tilde B_i)\eta}.
\end{equation}
\end{prop}
\begin{proof}
The formula \eqref{eq.DT} is just the Gram--Schmidt orthogonalization of the coherent algebra basis $\{I,A,A^2,\dots,A^d\}$. In formula \eqref{eq.DT2}, we replace $A^i$ by $AA_{i-1}$. But these two only differ by a linear combination of $A_0,A_1,\dots,A_{i-1}$ so the result of the orthogonalization procedure must be the same.
\end{proof}

\begin{prop}
Let $X$ be a quantum space and $A\colon l^2(X)\to l^2(X)$ an adjacency matrix of a simple cocommutative distance regular quantum graph. Denote its degree by $k$. Then there are sequences of numbers $(a_i)_{i=0}^d$, $(b_i)_{i=0}^{d-1}$, $(c_i)_{i=1}^d$ such that, for every $i=0,\dots,d$
\begin{equation}\label{eq.DT3}
AA_i=c_{i+1}A_{i+1}+a_iA_i+b_{i-1}A_{i-1}.
\end{equation}
Here, we use the convention $c_{d+1}A_{d+1}=0=b_{-1}A_{-1}$. These sequences satisfy
$$a_i+b_i+c_i=k,\quad c_i\neq0\qquad\text{for every $i$,}$$
$$a_0=0,\quad b_0=k,\quad c_1=1.$$
\end{prop}
\begin{proof}
We are going to prove this by induction. First, notice that $AA_0=AI=A=c_1A+a_0$, where $c_1=1$, $a_0=0$. Secondly, $\tilde\beta_2 A_2=AA_1-(AA_1)\bullet A_1 - (AA_1)\bullet A_0$. Since $A_0,\dots,A_d$ are minimal Schur projections, this indeed means that ${(A^2=)}AA_1=c_2A_2+a_1A_1+b_0A_0$. Schur-multiplying with $I=A_0$, we get $A^2\bullet I=b_0I$. Consequently, $b_0\eta^\dag\eta=\eta^\dag(A^2\bullet I)\eta=\eta^\dag(A\bullet A)\eta=\eta^\dag A\eta=k\eta^\dag\eta$, so $b_0=k$.

Now, pick any $i>1$ and suppose that formula \eqref{eq.DT3} holds for any index smaller. Again, by definition of $A_i$, we have $AA_i=\sum_{j=0}^{i+1}\lambda_{ij} A_j$, for some $\lambda_{ij}$. We need to show that $\lambda_{ij}=0$ if $j<i-1$. By induction hypothesis, we have $A_i\bullet AA_j=0$. Thus, $0=\eta^\dag(A_i\bullet AA_j)\eta=\eta^\dag(AA_i\bullet A_j)$. Since $AA_i\bullet A_j$ is a multiple of $A_j$, which is a Schur projection with $\eta^\dag A_j\eta\neq0$, this already implies that $0=AA_i\bullet A_j$, which is what we wanted to show.

It remains to prove that $a_i+b_i+c_i=k$. Denote $k_i$ the valency of $A_i$. Then
\begin{align*}
(a_i+b_i+c_i)k_i\eta^\dag\eta&=\eta^\dag(b_iA_i+a_iA_i+c_iA_i)\eta\\&=\eta^\dag(AA_{i+1}\bullet A_i+AA_i\bullet A_i+AA_{i-1}\bullet A_i)\eta\\&=\eta^\dag(A_{i+1}\bullet AA_i+A_i\bullet AA_i+A_{i-1}\bullet AA_i)\eta\\&=\eta^\dag(AA_i\bullet\sum_{j=0}^d A_j)\eta=\eta^\dag(AA_i\bullet J)\eta=\eta^\dag AA_i\eta=kk_i\eta^\dag\eta.
\end{align*}
\end{proof}

\begin{defn}
The sequence of numbers $(b_0,b_1,\dots,b_{d-1};c_1,c_2,\dots,c_d)$ from the previous proposition is called the \emph{intersection array} of a given distance regular quantum graph.
\end{defn}

\begin{prop}
The property of being distance regular as well as the intersection array are invariant with respect to quantum isomorphisms.
\end{prop}
\begin{proof}
Follows from the fact that a quantum isomorphism induces an isomorphism of the coherent algebra (preserving both products, both units, and both involutions).
\end{proof}

\subsection{Hadamard graphs}
\label{secc.Hadamard}

\begin{defn}[\cite{GroQHad}]
Let $X$ be a quantum space with $N=\eta^\dag\eta$. A \emph{quantum Hadamard matrix} is a linear map $H\colon l^2(X)\to l^2(X)$ such that 
$$H=H^*,\qquad H\bullet H=\eta\eta^\dag,\qquad HH^\dag=N\,\id=H^\dag H.$$
\end{defn}

\begin{ex}[\cite{GroQHad}]\label{E.transpose}
Consider the finite quantum space $X=M_n$, that is, $C(X)=M_n(\C)$, the counit is given by $\eta^\dag=n\Tr$, so $N:=\eta^\dag\eta=n^2$. Then the map $H\colon M_n(\C)\to M_n(\C)$ acting by $H(a)=n\,a^{\rm T}$, where $a^{\rm T}$ is the matrix transposition, is a quantum Hadamard matrix.
\end{ex}

\begin{defn}[\cite{GroQHad}]
Let $X$ be a quantum space. Denote $Y:=X\sqcup X\sqcup X\sqcup X$, that is,
$$C(Y)=C(X)\oplus C(X)\oplus C(X)\oplus C(X)=\C^4\otimes C(X).$$
Let $H\colon l^2(X)\to l^2(X)$ be a quantum Hadamard matrix. We define its associated \emph{quantum Hadamard graph} on $Y$ through the adjacency matrix $A\colon l^2(Y)\to l^2(Y)$
$$A=\begin{pmatrix}
0&0&H^+&H^-\\
0&0&H^-&H^+\\
H^{+\dag}&H^{-\dag}&0&0\\
H^{-\dag}&H^{+\dag}&0&0
\end{pmatrix},
$$
where $H^{\pm}=\frac{1}{2}(J\pm H)$.
\end{defn}

\begin{prop}
Let $X$ be a quantum space of dimension $N:=\eta^\dag\eta\neq 3$ and consider a quantum Hadamard matrix $H$ on $X$. Then the associated quantum Hadamard graph is distance regular with diameter four. Its intersection array equals to $(N,N-1,N/2,1;1,N/2,N-1,N)$.
\end{prop}
\begin{proof}
A straightforward computation shows that
\begin{align*}
AA_1&=N I+N/2\, A_2\\
AA_2&=(N-1)A_1+(N-1)A_3\\
AA_3&=N/2\,A_2+NA_4\\
AA_4&=A_3
\end{align*}
with
$$A_1=A,\quad
A_2=
\begin{pmatrix}
J-I&J-I&0&0\\
J-I&J-I&0&0\\
0&0&J-I&J-I\\
0&0&J-I&J-I
\end{pmatrix},
$$$$
A_3=
\begin{pmatrix}
0&0&H^-&H^+\\
0&0&H^+&H^-\\
H^{-\dag}&H^{+\dag}&0&0\\
H^{+\dag}&H^{-\dag}&0&0
\end{pmatrix},\quad
A_4=
\begin{pmatrix}
0&I&0&0\\
I&0&0&0\\
0&0&0&I\\
0&0&I&0
\end{pmatrix}.
$$
\end{proof}

\begin{rem}
This fact is well known for classical Hadamard matrices and graphs. Actually, a classical graph is a Hadamard graph corresponding to an $N\times N$ Hadamard matrix if and only if it is distance regular with intersection array $(N,N-1,\ppen N/2,1;\ppen 1,N/2,\ppen N-1,N)$ \cite[Theorem~1.8.1]{BCN89}. If the dimension $\eta^\dag\eta$ of the quantum space equals to some $N\in\N$ such that there exists a classical $N\times N$ Hadamard matrix, then our proposition follows from this classical result and from the fact that all quantum Hadamard graphs of a given size are quantum isomorphic \cite{GroQHad}.

On the other hand, if we choose $N$ such that there is no Hadamard matrix of this size, then the quantum Hadamard graph cannot be quantum isomorphic to any classical graph. For instance, we can take the transposition example (Ex.~\ref{E.transpose}) for $n>2$ odd. This shows that it makes sense to study cocommutative association schemes. Even though they are cocommutative, they bring new examples that have no classical counterpart.

It is interesting also from the viewpoint of quantum graphs. It disproves the converse of our considerations in \cite[Section~8.2]{GQGraph}, where we constructed a quantum graph which is not quantum isomorphic with any classical one by ensuring that its coherent algebra is not cocommutative. Here, we show that there may exist quantum graphs whose coherent algebra is cocommutative, but they are, nevertheless, not quantum isomorphic to any classical ones.
\end{rem}

\subsection{Strongly regular quantum graphs}
\begin{defn}
Let $X$ be a finite quantum space and let $A\colon l^2(X)\to l^2(X)$ define a simple quantum graph. We say that the quantum graph is \emph{strongly regular} with parameters $(n,k,\lambda,\mu)$ if $\eta^\dag\eta=n$, $A^2=kI+\lambda A+\mu (J-I-A)$.
\end{defn}

First, a couple of observations explaining the motivation, which directly generalize the classical properties.

\begin{prop}
A strongly regular (simple) quantum graph with parameters $(n,k,\lambda,\mu)$ is $k$-regular.
\end{prop}
\begin{proof}
We can compute $A^2\bullet I=kI$ (where we use the fact that $J\bullet I=I$). Then $A\eta=(A\bullet A)\eta=(A^2\bullet I)\eta=k\eta$ (in the second equality, we used $A=A\Ctrans$).
\end{proof}

\begin{prop}
A simple quantum graph is strongly regular if and only if $\spanlin\{I,J,A\}$ is a coherent algebra.
\end{prop}
\begin{proof}
By definition of a simple quantum graph $\spanlin\{I,J,A\}$ is closed under the Schur product and the involutions. By definition of strong regularity, it is closed under composition. Conversely, in any such coherent algebra, we can introduce parameters $k$, $\lambda$, and $\mu$ such that $A^2=kI+\lambda A+\mu(J-I-A)$.
\end{proof}

\begin{prop}
The complement of a strongly regular graph with parameters $(n,k,\lambda,\mu)$ is strongly regular with parameters $(n,n-1-k,n-2k+\mu-2,n-2k+\lambda)$
\end{prop}
\begin{proof}
Setting $B:=J-I-A$, we can easily compute
\[B^2=(n-1-k)I+(n-2k+\mu-2)B+(n-2k+\lambda)A.\qedhere\]
\end{proof}

\begin{prop}
\label{P.SRGqiso}
The property of being strongly regular as well as the corresponding parameters are invariant with respect to quantum isomorphisms
\end{prop}
\begin{proof}
Again follows from the fact that a quantum isomorphism induces an isomorphism of the coherent algebra.
\end{proof}

The idea behind strongly regular graphs is of course that they are the distance regular graphs with diameter two. This is true apart from a couple of exceptions.

\begin{ex}\label{E.srg}
For any finite quantum space $X$ with $n:=\eta^\dag\eta$, the \emph{complete graph} $K_X$ given by $A=J-I$ is strongly regular with parameters $(n,n-1,n-2,0)$ since $(J-I)^2=(n-2)(J-I)+(n-1)I$. Actually, since $A-J-I=0$, the parameter $\mu$ is undefined, but we can set it to zero. In this case, we see that $\spanlin\{I,A\}$ is a two-dimensional coherent algebra, so the graph is distance regular with diameter one.

Taking any $m\in\N$, we can consider the disjoint union of $m$ copies of $K_X$. That is, take $C(Y)=C(X)\oplus\cdots\oplus C(X)$ and an adjacency matrix $A\colon l^2(Y)\to l^2(Y)$ given by $A=(J_X-I_X)\oplus\cdots\oplus (J_X-I_X)$, where $J_X$ and $I_X$ are the corresponding operators on $l^2(X)$. Such a quantum graph is also strongly regular with parameters $(mn,n-1,n-2,0)$. But in this case the graph is not distance regular at all since the powers of $A$ do not generate $J_Y$. Also note that the graph is not connected.

Another extreme example is the empty graph $A=0$, which is strongly regular with parameters $(n,0,0,0)$. In this case, the $\lambda$ is actually not uniquely defined, but we can put it equal to zero.

Classically, this would exhaust all the examples with $\mu=0$, but there are more in the quantum setting. Take, for instance, the unique 1-regular simple quantum graph on $X=M_2$ (i.e. $C(X)=M_2(\C)$) from the classification in \cite{GQGraph} or \cite{Mat22}. This graph is known to be quantum isomorphic to the classical disjoint union $K_2\sqcup K_2$, so it has to be strongly regular with the same parameters, namely $(4,1,0,0)$.
\end{ex}

\begin{prop}\label{P.mu}
A strongly regular quantum graph is distance regular with diameter two if an only if $\mu\neq 0$ (excluding also $A=J-I$).
\end{prop}
\begin{proof}
Denote by $A$ the adjacency matrix. If the graph is supposed to be distance regular with diameter two, it means that $\spanlin\{I,J,A\}=\spanlin\{I,A,A^2\}$ is three-dimensional. Therefore, $J$ is not a linear combination of $I$ and $A$ and hence also $A^2$ must be a linear combination of $I$, $J$, and $A$ with a non-trivial coefficient at $J$. Thus, $\mu\neq0$.

For the converse, we claim that $A$ is not a linear combination of $I$ and $J$. If it was, the only possibilities would be $A=J-I$ and $A=0$ as $A$ must be a Schur projection. So, $\spanlin\{I,J,A\}$ must be three-dimensional. If, in addition, $\mu\neq0$, then $\spanlin\{I,A,A^2\}=\spanlin\{I,J,A\}$.
\end{proof}

\begin{prop}
A strongly regular quantum graph has $\mu=0$ if and only if it is disconnected or it is a complete graph $A=J-I$.
\end{prop}
\begin{proof}
If the quantum graph is disconnected, then by Prop.~\ref{P.drgConn} it cannot be distance regular and by Prop.~\ref{P.mu} has $\mu=0$. The complete graph has $\mu=0$ by Example~\ref{E.srg}. For the converse, assume that the strongly regular quantum graph has $\mu=0$ and it is connected. Note that $A=A^\dag$, so it is diagonalizable and hence the projection $\frac{1}{n}J$ on the eigenspace $\spanlin\{\eta\}$ must be a polynomial in $A$. Since $\mu=0$, it follows that actually $J$ is a linear combination of $I$ and $A$. As we mentioned already in the proof of Prop.~\ref{P.mu}, the only possibility is that $A=J-I$.
\end{proof}

So far, we were only able to construct strongly regular quantum graphs with $\mu=0$. One way to obtain an example with $\mu\neq 0$ is to construct the complement of a strongly regular graph with $\mu=0$. Such a quantum graph will naturally arise in Example~\ref{E.Knn2} by certain duality construction.

In Section~\ref{secc.Latin}, we are going to construct strongly regular graphs corresponding to certain quantum Latin squares. These have $\mu=6$. Below, we bring another example of a strongly regular quantum graph with $\mu\neq 0$:

\begin{ex}[Quantum rook's graph {\cite[Section~6.3]{GQGraph}}]
\label{E.rook}
Consider the quantum space $M_n$ defined by $C(M_n)=M_n(\C)$. Consider the adjacency matrix $A\colon M_n(\C)\to M_n(\C)$ defined by
$$e_{ij}\mapsto n\delta_{ij}e_{ii}-2e_{ij}+ \sum_{k} e_{i+k,i+j},$$
where the indices are taken mod $n$.

In \cite{GQGraph}, it was proven that this defines a simple quantum graph, which is quantum isomorphic to the classical $n\times n$ rook's graph. Classical $n\times n$ rook's graph is strongly regular with parameters $(n^2,2n-2,n-2,2)$. Consequently, the given quantum graph must be also strongly regular with the same parameters.
\end{ex}

As an open problem, we leave the task of finding a strongly regular quantum graph, which is not quantum isomorphic to any classical one.\footnote{Subsequent to submission of this paper, the problem was solved in~\cite{BHV26}.}

Finally, we can discuss eigenvalues of strongly regular quantum graphs. The situation here can actually be copied from the classical case. The adjacency matrix $A$ of a strongly regular quantum graph with parameters $(n,k,\lambda,\mu)$ has the eigenvector $\eta$ corresponding to the eigenvalue $k$. Consider some other eigenvector $x$ and denote by $p$ the corresponding eigenvalue. We surely have $Jx=0$ since $J$ is (up to normalization) the orthogonal projection on $\eta$. Consequently, the equation $A^2=kI+\lambda A+\mu (J-I-A)$ leads to $p^2x=kx+\lambda px-\mu(1+p)x$. Hence, $p$ satisfies
\begin{equation}\label{eq.SRGeigenvalue}
p^2+(\mu-\lambda)p+(\mu-k)=0.
\end{equation}
Assuming that $A$ defines a distance regular quantum graph of diameter two (i.e. $\mu\neq 0$ and $A\neq J-I$), so that $\{I,A,J-I-A\}$ is a 2-class CCQAS, it means that its eigenvalue matrix is given by
\begin{equation}\label{eq.SRGeigenmatrix}
P=\begin{pmatrix}
1&k&n-1-k\\
1&r&-1-r\\
1&s&-1-s
\end{pmatrix}
\end{equation}
where $r$ and $s$ are solutions of Equation~\eqref{eq.SRGeigenvalue}. (The first column contains eigenvalues of $I$, the second column eigenvalues of $A$, and the third one corresponds to $J-I-A$.)

\section{Translation quantum association schemes}
\label{sec.TQAS}

\subsection{Background in Hopf algebras}
In this section, we define finite quantum groups and recall the duality principle in this context.

\begin{defn}
A \emph{Hopf $*$-algebra} is a $*$-algebra $\Alg$ equipped with
\begin{itemize}
\item a $*$-homomorphism $\Delta\colon\Alg\to\Alg\otimes\Alg$ called the \emph{comultiplication}, which is supposed to be \emph{coassociative}, i.e. $(\Delta\otimes\id)\Delta=(\id\otimes\Delta)\Delta$,
\item a $*$-homomorphism $\epsilon\colon\Alg\to\C$ called the \emph{counit}, which satisfies $(\epsilon\otimes\id)\Delta=\id=(\id\otimes\epsilon)\Delta$,
\item a linear map $S\colon\Alg\to\Alg$ called the \emph{antipode}, which satisfies $m(S\otimes\id)\Delta=\eta\epsilon=m(\id\otimes S)\Delta$, where $m$ is the multiplication and $\eta$ is the unit of the algebra.
\end{itemize}
\end{defn}

A functional $\psi\colon\Alg\to\C$ on a Hopf $*$-algebra $\Alg$ is called \emph{left-invariant} if $(\psi\otimes\id)(\Delta(f))=\psi(f)\eta$. A positive left-invariant functional is called the \emph{Haar functional}. For finite-dimensional Hopf $*$-algebras, it always exists, it is given, up to normalization, uniquely, and it is always tracial \cite{VDa97}. As follows from \cite{Was23}, the Haar functional can be normalized\footnote{The normalization of the Haar functional is often chosen such that $\psi(\eta)=1$. This is, however, not compatible with the condition $mm^\dag=\id$ as we know that $\eta^\dag\eta=\dim\Alg$ in symmetric Frobenius $*$-algebras.} in such a way that it makes $\Alg$ a special symmetric Frobenius $*$-algebra. As in the preceding sections, we will denote the Haar functional defining the Frobenius structure on a Hopf $*$-algebra simply by $\eta^\dag$.

\begin{ex}
Suppose $\Gamma$ is a finite group. Then the $*$-algebra of all functions $C(\Gamma)=\{f\colon\Gamma\to\C\}$ has a structure of a Hopf $*$-algebra given by $\Delta(f)(x,y)=f(xy)$, $\epsilon(f)=f(e)$, $S(f)(x)=f(x^{-1})$. The Haar functional corresponds to the classical Haar measure, which, in the finite case, is the counting measure $\psi(f)=\sum_{x\in\Gamma}f(x)$.
\end{ex}

\begin{ex}
Suppose $\Gamma$ is a finite group. Then the group $*$-algebra $\C\Gamma$ generated by elements $\lambda_g$, $g\in\Gamma$ subject to relations $\lambda_g\lambda_h=\lambda_{gh}$, $\lambda_g^*=\lambda_{g^{-1}}$ is a Hopf $*$-algebra with respect to $\Delta(\lambda_g)=\lambda_g\otimes\lambda_g$, $\epsilon(\lambda_g)=1$, $S(\lambda_g)=\lambda_{g^{-1}}$. The Haar functional acts by $\psi(\lambda_g)=\delta_{eg}$.
\end{ex}

Since Hopf $*$-algebras are special symmetric Frobenius $*$-algebras, we can treat them as finite quantum spaces.

\begin{defn}
A \emph{finite quantum group} is a finite-dimensional Hopf $*$-algebra. We use the same terminology and notation as for finite quantum spaces. If we denote the finite quantum group by $\Gamma$, we use $C(\Gamma)$ to denote the underlying Hopf $*$-algebra and $l^2(\Gamma)$ the associated Hilbert space.
\end{defn}

\begin{lem}
Let $\Alg$ be a finite-dimensional Hopf $*$-algebra. Then the antipode satisfies the following properties.
\begin{enumerate}
\item It is $*$-preserving, self-adjoint and involutive: $S=S^*=S^\dag=S^{-1}$.
\item It is an algebra and coalgebra antihomomorphism: Denoting $\Sigma\colon x\otimes y\mapsto y\otimes x$, then $S\circ m=m\circ\Sigma\circ(S\otimes S)$ (i.e. $S(fg)=S(g)S(f)$) and $\Delta\circ S=(S\otimes S)\circ\Sigma\circ\Delta$.
\item It is unital and counital: $S\eta=\eta$, $\epsilon S=\epsilon$.
\end{enumerate}
\end{lem}
\begin{proof}
The fact that $S^2=\id$ follows from \cite{VDa97}. For the rest, see \cite{Tim08}.
\end{proof}

The inner product defined by the Haar functional allows us to construct a dual structure on a Hopf $*$-algebra exchanging the multiplication and comultiplication. We can define the multiplication $\Delta^\dag$ with unit $\epsilon^\dag$, the comultiplication $m^\dag$ with counit $\eta^\dag=\psi$. We can take the antipode $S=S^\dag$ and finally, to make it a Hopf $*$-algebra, we have to alter the involution a bit and define $f^\star=S f^*$. One can show that this is indeed a Hopf $\star$-algebra. We will call it the \emph{passive dual}.

\begin{ex}
\label{E.CGCG}
If $\Gamma$ is an ordinary finite group, then we can view any formal linear combination of elements in $\Gamma$ as a function on $\Gamma$ and vice versa. Hence, $C(\Gamma)$ can be identified with $\C\Gamma$ as vector spaces. We can denote the elements simply by $e_g$, $g\in\Gamma$. Taking this viewpoint, one can check that the Hopf $*$-algebra $\C\Gamma$ is the passive dual of $C(\Gamma)$ and vice versa.
\end{ex}

\begin{defn}
For any finite quantum group $\Gamma$, we denote by $\C\Gamma$ the passive dual of $C(\Gamma)$.
\end{defn}

The above definition is actually somewhat non-standard, so we will try to limit the usage of the notation $\C\Gamma$ to avoid confusions. In the literature, the following \emph{active} construction is more common:

\begin{prop}
Let $\Gamma$ be a finite-dimensional Hopf $*$-algebra. Then the dual vector space $\Alg^*$ is a Hopf $\star$-algebra with respect to the following:
\begin{itemize}
\item multiplication $(\phi,\psi)\mapsto\phi*\psi:=(\phi\otimes\psi)\circ\Delta$,
\item comultiplication $\phi\mapsto\phi\circ m$,
\item unit $\epsilon$,
\item counit $\phi\mapsto\phi(\eta)$
\item involution $\phi\mapsto\phi^\star:=\phi^*\circ S$, i.e. $\phi^\star(f)=\overline{\phi(S(f)^*)}$
\item antipode $\phi\mapsto\phi\circ S$.
\end{itemize}
\end{prop}

The two approaches of defining the Hopf dual are equivalent differing just by taking the adjoint. (As a matter of fact, the adjoint operation is antilinear, not linear. So, the two constructions actually do not produce isomorphic Hopf $*$-algebras in general. In fact, they differ by the order of the multiplication.)


\begin{nota}
\label{N.prods}
In the following text, we will prefer using the ``neutral'' notation $l^2(\Gamma)$ with no preferred product instead of $C(\Gamma)$ or $\C\Gamma$. We will denote by $m$ and $\Delta$ the product and coproduct of $C(\Gamma)$. For $x,y\in l^2(\Gamma)$ we will denote
\begin{itemize}
\item $x\bullet y=m(x\otimes y)$,
\item $x*y=\Delta^{\dag}(x\otimes y)$,
\item $x^*$ as in $C(\Gamma)$,
\item $x^\star=S(x^*)$.
\end{itemize}
A similar notation can be used also for functionals $l^2(\Gamma)\to\C$. Note also that we have the unit $\eta$ with respect to the product $\bullet$ and the unit $\epsilon^\dag$ with repsect to the product $*$. Hence, the tuple $(l^2(\Gamma),\bullet,*,{}^*,{}^\star)$ looks like a coherent algebra except that the elements are just abstract; they are not realized as linear operators on some $l^2(X)$. In Proposition~\ref{P.GHomog} we will prove that this ``coherent algebra'' is also homogeneous:
\end{nota}

\begin{ex}\label{E.GammaOp}
For a finite group $\Gamma$, denote by $\{e_g\}_{g\in\Gamma}$ the canonical basis of $l^2(\Gamma)$. Then
\begin{itemize}
\item $e_g\bullet e_h=\delta_{gh}e_g$ is the pointwise product,
\item $e_g*e_h=e_{gh}$ is the convolution,
\item $e_g^*=e_g$,
\item $e_g^\star=e_{g^{-1}}$.
\end{itemize}
We have the units $\eta=\sum_{g\in\Gamma}e_g$ (i.e.\ the all-one-vector) and $\epsilon^\dag=e_e$.
\end{ex}

Since the dual of $\C\Gamma$ is again a finite-dimensional Hopf algebra, it defines a new finite quantum group. We introduce the following terminology:

\begin{defn}
\label{D.dualG}
Let $\Gamma$ be a finite quantum group. We will say that $\hat\Gamma$ is \emph{dual} to $\Gamma$ if the Hopf $*$-algebra $C(\hat\Gamma)$ is isomorphic to the passive dual of the Hopf $*$-algebra $C(\Gamma)$. We will denote by $\Phi\colon l^2(\Gamma)\to l^2(\hat\Gamma)$ the canonical vector space isomorphism and call it the \emph{Fourier transform}.\footnote{For any finite quantum group $\Gamma$, we can \emph{construct} the dual $\hat\Gamma$ simply by taking the Hopf $*$-algebra $C(\hat\Gamma)$ to be equal to the passive dual of the Hopf $*$-algebra $C(\Gamma)$. Then $\l^2(\hat\Gamma)=l^2(\Gamma)$ and $\Phi$ is simply the identity.} We will extend it to a monoidal (not unitary) functor acting on the category of all tensors $A\colon l^2(\Gamma)^{\otimes k}\to l^2(\Gamma)^{\otimes l}$ by
$$\hat A=\Phi^{\otimes l}A\Phi^{-1\,\otimes k}\colon l^2(\hat\Gamma)^{\otimes k}\to l^2(\hat\Gamma)^{\otimes l}.$$
\end{defn}

In particular, we denote $\hat x=\Phi(x)$ for $x\in l^2(\Gamma)$. 
Also $\widehat{\epsilon^\dag}$ is the unit of $C(\hat\Gamma)$, $\widehat{\eta^\dag}$ is the counit, $\hat S$ is the antipode, $\widehat{\Delta^\dag}$ is the multiplication, and $\widehat{m^\dag}$ is the comultiplication. We will use Notation~\ref{N.prods} also for the dual. That is,
\begin{itemize}
\item $\hat x\bullet\hat y:=\widehat{\Delta^\dag}(\hat x\otimes\hat y)=\widehat{\Delta^\dag(x\otimes y)}=\widehat{x*y}$
\item $\hat x*\hat y:=\widehat{m^\dag}^\dag(\hat x\otimes\hat y)$
\item $\hat x^*:=\widehat{x^\star}$
\item $\hat x^\star:=\hat S\hat x^*=\hat S\widehat{x^\star}=\widehat{Sx^\star}=\widehat{x^*}$
\end{itemize}

At the moment, it is not obvious how to simplify the second definition and relate it with the operations in the original Hopf $*$-algebra. This is because, it is not clear whether the inner product and hence the adjoint operation in the dual coincides with the adjoint in the original algebra. We devote the rest of this section to explaining this relationship.


Recall the diagrammatic calculus that is used for Frobenius $*$-algebras or quantum spaces. The associativity of $m$, the unit $\eta$, the snake equation \eqref{eq.snake}, the ``special'' condition $mm^\dag=\id$, or the Frobenius law \eqref{eq.Flaw} mean that computing with these morphisms diagrammatically is extremely simple.

In Hopf algebras, we have the dual structure as well. We have a comultiplication $\Delta$, for which we can introduce the diagram $\wspider{1/2}$; we have the counit $\epsilon$, for which we can introduce the diagram $\wspider{1/0}$. Now, the coassociativity and counit laws mean that
$\Diagram{\Dmor{circ}1/2 (1,0) \Dmor{circ}1/2 (1.5,1) \draw (0.5,0.5) -- (0.5,1.5);}=
 \Diagram{\Dmor{circ}1/2 (2,0) \Dmor{circ}1/2 (1.5,1) \draw (2.5,0.5) -- (2.5,1.5);}$
and
$\Diagram{\Dmor{circ}1/2 (1,.5) \Dmor{circ}0/0 (1.5,1)}=\Did=
 \Diagram{\Dmor{circ}1/2 (1,.5) \Dmor{circ}0/0 (0.5,1)}$.

We can also introduce $\wspider{2/1}:=\Delta^\dag$ and $\wspider{1/0}:=\epsilon^\dag$. However, it is already not clear, whether all the laws we mentioned above (like the special condition or the Frobenius law) have also a ``white'' counterpart. But since the dual of a Hopf $*$-algebra is a Hopf $*$-algebra again, it must also define a dual quantum space, where all these laws should again hold. The only subtlety is that we defined $\wspider{2/1}$ and $\wspider{1/0}$ using the dagger $\dag$ in the original Hopf algebra, not in the dual. So, again, the only thing which remains unclear is whether the inner product in the dual coincides with the inner product in the original Hopf algebra. It turns out that this is true up to a multiplicative constant:

\begin{prop}\label{P.dagdual}
Let $\Gamma$ be a finite quantum group and $\hat\Gamma$ its dual. Then, for any $A\colon l^2(\Gamma)^{\otimes k}\to l^2(\Gamma)^{\otimes l}$, we have
$$\widehat{A^\dag}=n^{k-l}\hat A^\dag,\quad\widehat{A\Ctrans}=n^{k-l}\hat S^{\otimes l}\hat A\Ctrans\hat S^{\otimes k},\quad\widehat{A^*}=\hat S^{\otimes l}\hat A^*\hat S^{\otimes k}.$$
\end{prop}
We will prove the statement by a series of lemmata. In all the statements, we assume that $\Alg$ is a finite-dimensional Hopf $*$-algebra.
\begin{lem}
\label{L.epsilon}
For any $x\in\Alg$,
$$(\epsilon\otimes\id)m^\dag(x)=\epsilon(x)\epsilon^\dag$$
Consequently, $\epsilon$ is left-invariant in the passive dual.
\end{lem}
\begin{proof}
We use the Frobenius law, the fact that $\epsilon$ is a homomorphism, so $\epsilon m=\epsilon\otimes\epsilon$ and that $\epsilon=\epsilon^*$, so $\epsilon\Ctrans=\epsilon^\dag$:
$$
\Diagram{\DMor{vec}0/1 (1,-.5) {$\scriptstyle x$}
         \Dmor{bcirc}1/2 (1,1) \Dmor{circ}0/0 (0.5,1.5)}=
\Diagram{\DMor{vec}0/1 (0.5,-.5) {$\scriptstyle x$}
         \Dmor{bcirc}2/1 (1,1) \Dmor{circ}0/0 (1,1.5) \Dmor{bcirc}0/2 (2,0) \draw (2.5,0.5) -- (2.5,1.5);}=
\Diagram{\DMor{vec}0/1 (1,-.5) {$\scriptstyle x$} \Dmor{circ}1/0 (1,1)
         \Dmor{circ}1/0 (2,1) \Dmor{bcirc}0/2 (2.5,0) \draw (3,0.5) -- (3,1);}=
\Diagram{\DMor{vec}0/1 (1,-.5) {$\scriptstyle x$} \Dmor{circ}1/0 (1,1)
         \Dmor{circ}0/1 (2.5,0) \draw (2.5,0.5) -- (2.5,1);}
$$
\end{proof}
\begin{lem}
For any $x,y\in\Alg$,
$$\epsilon(x*y)=\eta^\dag(S(x)\cdot y)=\eta^\dag(x\cdot S(y)).$$
Consequently, the bilinear form induced by $\epsilon$ in the passive dual differs from the original by composing with the antipode:
${\wspider{2/0}}=
\Diagram{\DMor{square}1/1 (1,0) {$\scriptstyle S$} \Dmor{bcirc}2/0 (1.5,1.5) \draw (2,1) -- (2,-1);}=
\Diagram{\DMor{square}1/1 (2,0) {$\scriptstyle S$} \Dmor{bcirc}2/0 (1.5,1.5) \draw (1,1) -- (1,-1);}$.
\end{lem}
\begin{proof}
See \cite{LS69}.
\end{proof}
\begin{lem}
The sesquilinear form $(x,y)\mapsto\epsilon(x^\star *y)$ coincides with the inner product in $\Alg$. In particular, it is again a positive inner product. Hence, $\epsilon$ is a positive Haar functional.
\end{lem}
\begin{proof}
This is a direct consequence of the preceding lemma and the definition of the involution in the dual:
$$\epsilon(x^\star*y)=\eta^\dag(S(x^\star)y)=\eta^\dag(x^*y).$$
\end{proof}
\begin{lem}
We have $\Delta^\dag\Delta=n\,\id$, where $n=\eta^\dag\eta$. That is, we get the structure of a special Frobenius $*$-algebra on the dual if we scale the above mentioned Haar functional, bilinear form, and inner product by the factor $\frac{1}{n}$.
\end{lem}
\begin{proof}
Since $\epsilon$ is a Haar functional, we know that it can be normalized such that the corresponding Frobenius $*$-algebra is special. Hence $\Delta^\dag\Delta=\alpha\,\id$ for some $\alpha\in\C$. Now, it remains to determine the constant $\alpha$. For this, we use the fact that $\Delta$ is unital:
$$\alpha n=\alpha\,\eta^\dag\eta=\eta^\dag\Delta^\dag\Delta\eta=\eta^\dag\eta\,\eta^\dag\eta=n^2.$$
\end{proof}
\begin{proof}[Proof of Proposition~\ref{P.dagdual}]
We just proved that the inner product for $\Gamma$ coincides with the inner product for $\hat\Gamma$ up to the constant $1/n$. From this, the formula $\widehat{A^\dag}=n^{k-l}\hat A^\dag$ directly follows. We also proved that the bilinear form differs in addition by composing with the antipode $S$. From this, the formula for the transposition follows. Finally, combining these two, we arrive at the third formula.
\end{proof}

\begin{prop}\label{P.FourierOps}
Let $\Gamma$ be a finite quantum group and $\hat\Gamma$ its dual. Then, for every $x,y\in l^2(\Gamma)$, we have
\begin{align*}
\widehat{x*y}&=\hat x\bullet \hat y,&n\,\widehat{x\bullet y}&=\hat x*\hat y,\\
\widehat{x^\star}&=\hat x^*,&\widehat{x^*}&=\hat x^\star.
\end{align*}
\end{prop}
(Compare with Definition~\ref{D.dual} of duality between quantum coherent algebras.)
\begin{proof}
It remains to show the second formula, everything else was already proven below Definition~\ref{D.dual}. For that, we use the first formula of Proposition~\ref{P.dagdual}. Recall that the comultiplication in the dual is given by $\widehat{m^\dag}$.
\[n\,\widehat{x\bullet y}=n\,\widehat{m(x\otimes y)}=n\,\hat m(\hat x\otimes\hat y)=(\widehat{m^\dag})^\dag(\hat x\otimes\hat y)=\hat x*\hat y\qedhere\]
\end{proof}

\begin{rem}
For any ordinary finite abelian group $\Gamma$, we have that $\hat\Gamma\simeq\Gamma$. This is because for every $g\in\Gamma$, there is a character $\tau_g\in l^2(\Gamma)$ such that $\tau_g\tau_h=\tau_{gh}$. In this context, by \emph{Fourier transform} we usually mean the linear mapping $e_g\mapsto\tau_g$, which then also maps $\tau_g\mapsto \frac{1}{n}e_g$.
\end{rem}

\begin{prop}
\label{P.GHomog}
Let $\Gamma$ be a finite quantum group. Then $\epsilon^\dag$ is a minimal projection in $C(\Gamma)$.
\end{prop}
\begin{proof}
Consider any $x\in l^2(X)$. Then it holds that $x\bullet\epsilon^\dag=\epsilon(x)\epsilon^\dag$ -- we essentially proved this already in Lemma~\ref{L.epsilon}. Putting $x=\epsilon^\dag$, this means that $\epsilon^\dag$ is indeed a projection. In addition, if $x$ is a projection, then the equation means that $x\bullet\epsilon^\dag$ equals either to $\epsilon^\dag$ or to~0.
\end{proof}

\subsection{Classical translation association schemes}
First, we recall the classical theory, see also e.g. \cite[Section~6]{MT09} or \cite[Section~2.10.B]{BCN89} for a summary.

A translation association scheme is an association scheme over a group $\Gamma$ such that for each $A_i$, we have $[A_i]_{xy}=1$ implies $[A_i]_{gx,gy}=1$ for any $g,x,y\in\Gamma$. Defining $\Gamma_i=\{g\in\Gamma\mid A_{ge}=1\}$, we obtain a partition $\{\Gamma_i\}$ of $\Gamma$ and each $A_i$ is a Cayley graph of $\Gamma$ with respect to the set $\Gamma_i$.

A particular example is the \emph{group scheme}, where the elements of the scheme are indexed by the group elements and we have $[A_g]_{xy}=1$ if and only if $x=gy$. Note that in this case, the operations of $\bullet$ and $*$ in $l^2(\Gamma)$ exactly correspond to the Schur product $\bullet$ and composition for the adjacency matrices. Indeed, we have $A_gA_h=A_{gh}$, and $A_g\bullet A_h=\delta_{gh}A_g$, which exactly corresponds to Example~\ref{E.GammaOp}. The same holds for the involutions. Note in addition, that $A_g$ actually acts by convolution, i.e. $A_ge_h=e_g*e_h$.

Taking arbitrary partition $\{\Gamma_i\}$ of a group $\Gamma$, the associated Cayley graphs $A_i=\sum_{g\in\Gamma_i}A_g$ define an association scheme if and only if they span a coherent algebra. That is, translation association schemes are exactly the \emph{subschemes} of the corresponding group scheme. Here, a subscheme means that the associated Bose--Mesner algebra is a coherent subalgebra.

Equivalently, the partition $\{\Gamma_i\}$ defines a translation association scheme if and only if $\spanlin\{\alpha_i\}$ with $\alpha_i=\sum_{g\in\Gamma_i}e_g\in l^2(\Gamma)$ is closed under both products and involutions. In fact, any subalgebra $R\subset l^2(\Gamma)$ that is closed under both products and involutions, will have such a basis and define a translation association scheme. Such a subalgebra is called a \emph{Schur ring}. Note again that the Schur ring is isomorphic with the corresponding coherent algebra and that all the adjacency matrices act via convolution $A_ix=\alpha_i*x$.

The duality we want to talk about here was first formulated by Tamaschke \cite{Tam63} for Schur rings and later reformulated by Delsarte \cite{Del73} to the language of association schemes.

So, assume now that the group $\Gamma$ is abelian. Denote by $\tau_g\in l^2(\Gamma)$ the irreducible characters. It is known that these are eigenvectors of the adjacency matrix corresponding to any Cayley graph, so $A_i\tau_g=\lambda_g^i\tau_g$, where\footnote{In the combinatorial literature, the formula usually does not contain the inverse. This is because it is common to use a transposed definition of a Cayley graph and translation schemes, where the group acts from the right.} $\lambda^i_g=\sum_{k\in\Gamma_i}\tau_g(k)^{-1}$. 

Denote by $V_j\subset l^2(\Gamma)$ the common eigenspaces of all $A_i$; in particular, $V_0=\spanlin\{\tau_e\}$, where $\tau_e=\eta$. Denote by $E_j$ the corresponding projections and by $\{\hat\Gamma_j\}$ the partition of $\hat\Gamma=\Gamma$ such that $\{\tau_g\}_{g\in\hat\Gamma_j}$ forms a basis of $V_j$. We claim that $\{\hat\Gamma_i\}$ defines a translation association scheme that is dual to the original one.

This is clear from the Schur ring viewpoint. If we denote $\epsilon_j:=\sum_{g\in\hat\Gamma_j}\tau_g$, we obtain the basis of convolution projections of the original Schur ring corresponding to the projections $E_j$. Now, applying the Fourier transform on the Schur ring, we obviously get a dual one, where the products are swapped and it has a basis of Schur projections $\beta_j=\sum_{g\in\hat\Gamma_j}e_g$.

\begin{ex}
\label{E.Knn}
Take $\Gamma=\Z_n\times\Z_2$ and its generating set $S=\{(i,1)\mid i\in\Z_n\}$. As one can easily check, the Cayley graph of $\Gamma$ with respect to $S$ is the complete bipartite graph $K_{n,n}$. Denote by $A$ its adjacency matrix. The graph is actually strongly regular since $A^2=nJ-nA$. That is, we have a translation association scheme $\Ass=\{A_0,A_1,A_2\}$ with
$$A_0=I,\quad A_1=A,\quad A_2=J-I-A$$
corresponding to the partition of $\Gamma$ given by
$$\Gamma_0=\{(0,0)\},\quad\Gamma_1=\{(i,1)\mid i\in\Z_n\},\quad\Gamma_2=\{(i,0)\mid i\in\Z_n\setminus\{0\}\}.$$

Since all $A_i$ are polynomials in $A$, determining the common eigenspaces reduces to studying just $A$. Consider the eigenvector $\tau_{(i,j)}$, $\tau_{(i,j)}(a,b)=\omega^{ia}(-1)^{jb}$, where $\omega$ is some primitive $n$-th root of unity. We determine its eigenvalue using the formula
$$\lambda_{(i,j)}=\sum_{(a,b)\in S}\tau_{(i,j)}(a,b)^{-1}=(-1)^j\sum_{a=0}^{n-1}\omega^{ia}=
\begin{cases}
0&i\neq 0,\\
n&i=0,j=0,\\
-n&i=0,j=1.
\end{cases}$$

Thus, we have three eigenspaces corresponding to the eigenvalues $n$, 0, and $-n$. Denote them by $V_0$, $V_1$, and $V_2$. We have $V_i=\{\tau_g\mid g\in\hat\Gamma_i\}$, here
$$\hat\Gamma_0=\{(0,0)\},\quad\hat\Gamma_1=\{(i,j)\mid i\neq0\},\quad\hat\Gamma_2=\{(0,1)\}.$$

That is, the dual of $K_{n,n}$ as a strongly regular Cayley graph of $\Gamma$ with respect to the generating set $S=\Gamma_1$, is a strongly regular Cayley graph of $\hat\Gamma=\Gamma$ with respect to the generating set $\hat\Gamma_1$. As one can easily find out, it is the cocktail party graph (complete $n$-partite graph $K_{2,2,\dots,2}$; the complement of $K_2\sqcup\cdots\sqcup K_2$).
$$
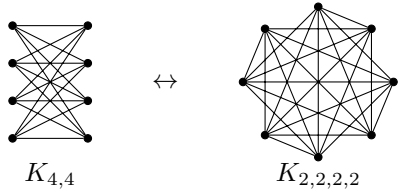
\begin{tikzpicture}[baseline={(0,-.75cm-.5ex)},every node/.style={circle,draw,fill=black,inner sep=1pt}]
\node (A1) at (0,0){};\node (B1) at (0,-.5){};\node (C1) at (0,-1){};\node (D1) at (0,-1.5){};
\node (A2) at (1,0){};\node (B2) at (1,-.5){};\node (C2) at (1,-1){};\node (D2) at (1,-1.5){};
\draw (A1) -- (A2) -- (B1) -- (B2) -- (C1) -- (C2) -- (D1) -- (D2) -- (C1) -- (A2) -- (D1) -- (B2) -- (A1) -- (C2) -- (B1) -- (D2) -- (A1);
\node[fill=white,draw=none] at (0.5,-2) {$K_{4,4}$};
\end{tikzpicture}
\qquad\leftrightarrow\qquad
\begin{tikzpicture}[baseline={(0,-.5ex)},every node/.style={circle,draw,fill=black,inner sep=1pt}]
\foreach \i in {0,1,...,7} {
	\node (\i) at (\i*45:1) {};
}
\draw (0) -- (4) -- (3) -- (7) -- (0) -- (2) -- (6) -- (5) -- (1) -- (2) -- (4) -- (6) -- (0) -- (3) -- (6) -- (1) -- (3) -- (5) -- (7) -- (1) -- (4) -- (7) -- (2) -- (5) -- (0);
\node[fill=none,draw=none] at (0,-1.25) {$K_{2,2,2,2}$};
\end{tikzpicture}
$$

We can also compute the eigenvalues of the other elements of the association scheme in order to obtain the eigenvalue matrix. Obviously, $A_0=I$ has a single eigenvalue 1, whereas for $A_2=\frac{1}{n}A^2-I$, the eigenspaces $V_0$ and $V_2$ have the eigenvalue $n-1$ while $V_1$ has the eigenvalue $-1$. Hence, the eigenvalue matrix is given by (cf.\ also Eq.~\eqref{eq.SRGeigenmatrix})
$$P=\begin{pmatrix}
1& n&n-1\\
1& 0& -1\\
1&-n&n-1
\end{pmatrix}.$$
Inverting this matrix, we can compute the dual eigenvalue matrix $Q=2nP^{-1}$ as
$$Q=\begin{pmatrix}
1&2n-2& 1\\
1&   0&-1\\
1&  -2& 1
\end{pmatrix}.$$
Denote by $E_i$ the projection on the eigenspace $V_i$. Using the above computation, we can express the adjacency matrices $A_i$ as a combination of the projections $E_i$ and vice versa:
$$A_0=E_0+E_1+E_2,\quad A_1=nE_0-nE_2,\quad A_2=(n-1)(E_0+E_2)-E_1,$$
$$E_0=\frac{1}{2n}(A_0+A_1+A_2)=J,\quad E_1=\frac{1}{n}((n-1)A_0-A_2),\quad E_2=\frac{1}{2n}(A_0-A_1+A_2).$$
Doing the same computation for the dual graph, we can check that the eigenvalues and dual eigenvalues get swapped.
\end{ex}

\begin{ex}
We recently studied quantum symmetries and deformations of some distance regular Cayley graphs of abelian groups in \cite{GroAbSym,GQGraph}. We computed the eigenspaces in \cite{GroAbSym} from which one can easily see that
\begin{itemize}
\item the halved hypercube graph is distance regular and its dual is the folded hypercube graph and vice versa;
\item all the Hamming graphs (including the hypercube graph or the rook's graph) are distance regular and self-dual.
\end{itemize}
\end{ex}

\subsection{Twisted translation schemes}
In this section, we are going to show how to construct a pair of dual CCQAS based on a certain twisting procedure formulated in \cite[Section~5]{GQGraph}. That is, we start with an ordinary abelian group $\Gamma$, construct a pair of translation association schemes, and twist them. The duality property should survive this twisting.

So, consider a classical translation association scheme over an abelian group $\Gamma$ corresponding to some partition $\{\Gamma_i\}$. Denote by $\lambda^i_g$ the eigenvalues so that $A_i\tau_g=\lambda^i_g\tau_g$.

Now, consider a bicharacter $\sigma$ on $\Gamma$. Consider the comodule algebra twist
$$C(\breve\Gamma)=C^*(\breve\tau_g,g\in\Gamma\mid\breve\tau_g\breve\tau_h=\bar\sigma_{g,h}\breve\tau_{gh}).$$
We define the twisted translation association scheme by $\breve A_i\breve\tau_g=\lambda^i_g\breve\tau_g$.

We should first prove that this is indeed a CCQAS. So, let us go through the axioms. Axiom (1) means that $\breve A_i$ define quantum graphs -- this was proven in \cite{GQGraph}. All the other axioms follow from the fact that in the eigenbasis, $A_i$ and $\breve A_i$ are the same matrices. For (3), note also that $\eta=\tau_e$ in $C(\Gamma)$ and $\breve\eta=\breve\tau_e$ in $C(\breve\Gamma)$, so the matrix $J=\eta\eta^\dag$ also looks the same in both cases.

Now, by definition, the eigenvalue matrix $P$ does not change and hence also the dual eigenvalue matrix $Q$. So, we obtain a CCQAS with isomorphic coherent algebra to the original translation association scheme. Hence, if we twist the dual, we must obtain a dual of the twist.

As a remark, note that it already follows from \cite[Section~5]{GQGraph} that the resulting CCQAS will be actually quantum isomorphic to the original one in the sense that there is a monoidal unitary isomorphism of the categories $\langle m,\eta,A_1,\dots,A_d\rangle\to\langle\breve m,\breve\eta,\breve A_1,\dots,\breve A_d\rangle$ mapping generators to generators.

Nevertheless, this example might be slightly disappointing. We got an association scheme that is formally dual to the original one, but the group duality that was originally lying in the background disappeared. The comodule algebra twist $C(\breve\Gamma)$ does not have any Hopf algebra structure, so the twisted association scheme cannot be interpreted as a quantum translation association scheme belonging to some quantum group. Examples of this latter kind will be constructed in the following section.

\begin{ex}
The duality properties of the hypercube graph, the halved and folded hypercube, or the rook's graph pass to their twists, which were explicitly constructed in \cite[Sections~6,~7]{GQGraph}. We mentioned the quantum rook's graph already in Example~\ref{E.rook}.
\end{ex}


\subsection{Translation quantum association schemes}
\label{secc.trans}

In this section, we formulate the main theorem of this article concerning the duality of translation quantum association schemes. For that purpose, we would like to be able to speak about general quantum association schemes, not just the cocommutative ones. As already indicated, this is possible only by studying the associated coherent algebra. So, similarly to finite quantum spaces and finite quantum groups, quantum association schemes will be just fictive objects that will be given by their Bose--Mesner algebra.

\begin{defn}
A $d$-class \emph{quantum association scheme} over a finite quantum space $X$ is a $(d+1)$-dimensional homogeneous quantum coherent algebra over $X$. Similarly to finite quantum spaces and finite quantum groups, we will use a special notation and terminology here. We will usually denote a quantum association scheme by $\Ass$ and treat it as some abstract object. The associated quantum coherent algebra will then be denoted by $\spanlin\Ass$ and called its \emph{Bose-Mesner algebra}.
\end{defn}

Recall Theorem~\ref{T.AlgAss}, which tells us that if $\spanlin\Ass$ is cocommutative, then there is an actual CQAS $\Ass$ such that $\spanlin\Ass$ is its Bose--Mesner algebra.

\begin{defn}
Let $\Gamma$ be a finite quantum group. For any $x\in l^2(\Gamma)$, we denote $A_x:=\Delta^\dag(x\otimes\id)$, that is, $A_xy=x*y$. We interpret $A_x$ as the \emph{weighted Cayley graph} of $\Gamma$ with respect to $x$.
Graphically,
$$
\Diagram{\DMor{square}1/1 (1,0.5) {$\scriptstyle A_x$}}=
\Diagram{\Dmor{circ}2/1 (1,1) \DMor{vec}0/1 (0.5,-0.5) {$\scriptstyle x$} \draw (1.5,0.5) -- (1.5,-0.5);}
$$
\end{defn}

\begin{prop}
Let $\Gamma$ be a finite quantum group. Then
\begin{enumerate}
\item $A_xA_y=A_{x*y}$,
\item $A_x\bullet A_y=A_{x\bullet y}$,
\item $A_x^\dag=A_{x^\star}$,
\item $A_x^*=A_{x^*}$.
\end{enumerate}
\end{prop}
\begin{proof}
We do the proof using the graphical calculus. For (1), we use associativity of $\Delta^\dag$:
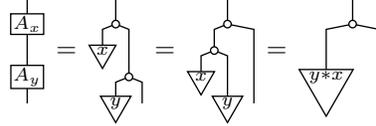
$$
\Diagram{\DMor{square}1/1 (1,1.5) {$\scriptstyle A_x$} \DMor{square}1/1 (1,-.5) {$\scriptstyle A_y$}}=
\Diagram{\Dmor{circ}2/1 (1,1.5) \Dmor{circ}2/1 (1.5,0) \draw (1,2) -- (1,2.5);
         \DMor{vec}0/1 (0.5,0) {$\scriptstyle x$} \draw (1.5,1) -- (1.5,0.5);
         \DMor{vec}0/1 (1,-1.5)  {$\scriptstyle y$} \draw (2,-0.5) -- (2,-1.5);}=
\Diagram{\Dmor{circ}[-0-/-] (2,1.5) \Dmor{circ}2/1 (1,0.5) \draw (2,2) -- (2,2.5);
         \DMor{vec}0/1 (0.5,-1) {$\scriptstyle x$} \draw (3,1) -- (3,-1.5);
         \DMor{vec}0/1 (1.5,-1.5)  {$\scriptstyle y$} \draw (1.5,0) -- (1.5,-0.5);}=
\Diagram{\Dmor{circ}[-0-/-] (2,1.5) \draw (2,2) -- (2,2.5);
         \DMor{vec}0/1 (1,0) {$\scriptstyle x*y$} \draw (3,1) -- (3,-1.5);}
$$

For (2), we are going to use the fact that $\Delta$ is a homomorphism. Graphically, this can be written as
$$
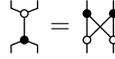
\Diagram{\Dmor{circ}1/2 (1,1) \Dmor{bcirc}2/1 (1,0)}=
\Diagram{\draw (1,1) -- (2,0); \draw (2,1) -- (1,0);
         \Dmor{circ}1/1 (1,0) \Dmor{circ}1/1 (2,0)
         \Dmor{bcirc}1/1 (1,1) \Dmor{bcirc}1/1 (2,1)}
$$
Also recall that the white and the black bilinear form differ by the antipode and that the antipode is an antihomomorphism. This, in particular, means that
$$
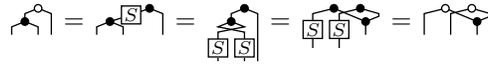
\Diagram{\Dmor{circ}2/0 (1.5,1) \Dmor{bcirc}2/1 (1,0) \draw (2,0.5) -- (2,-0.5);}=
\Diagram{\Dmor{bcirc}2/0 (1.5,2) \draw (2,1.5) -- (2,-1.5);
         \DMor{square}1/1 (1,0.5) {$\scriptstyle S$}
         \Dmor{bcirc}2/1 (1,-1)}=
\Diagram{\Dmor{bcirc}2/0 (1.5,2) \draw (2,1.5) -- (2,-1.5);
         \Dmor{bcirc}2/1 (1,1)
         \DMor{square}[-/00-] (0.5,-0.5) {$\scriptstyle S$}
         \DMor{square}[-/-00] (1.5,-0.5) {$\scriptstyle S$}}=
\Diagram{\DMor{square}1/1 (1,0) {$\scriptstyle S$} \Dmor{bcirc}[-0-/] (2,1.5)
         \DMor{square}1/1 (2,0) {$\scriptstyle S$} \Dmor{bcirc}[-0-/] (3,1.5)
         \DMor{bcirc}1/2 (3.5,0) {}}=
\Diagram{\Dmor{circ}[-0-/] (2,1) \draw (1,0.5) -- (1,-0.5);
         \Dmor{circ}[-0-/] (3,1) \draw (2,0.5) -- (2,-0.5);
         \Dmor{bcirc}1/2 (3.5,0) {}}$$
From this, we can derive:
$$
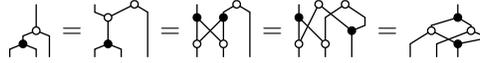
\Diagram{\Dmor{circ}2/1 (1.5,1) \Dmor{bcirc}2/1 (1,0) \draw (2,0.5) -- (2,-0.5);}=
\Diagram{\Dmor{circ}1/2 (1,1) \Dmor{bcirc}2/1 (1,0)
         \Dmor{circ}2/0 (2,2) \draw (2.5,1.5) -- (2.5,-0.5);}=
\Diagram{\draw (1,1) -- (2,0); \draw (2,1) -- (1,0);
         \Dmor{circ}1/1 (1,0) \Dmor{circ}1/1 (2,0)
         \Dmor{bcirc}1/1 (1,1) \Dmor{bcirc}1/1 (2,1)
         \Dmor{circ}2/0 (2.5,2) \draw (3,1.5) -- (3,-0.5);}=
\Diagram{\draw (1,1) -- (2,0); \draw (2.5,1) -- (1,0);
         \Dmor{circ}1/1 (1,0) \Dmor{circ}[-/0-] (2,0) \Dmor{bcirc}1/1 (1,1)
         \Dmor{circ}[00-/] (2.5,1) \Dmor{circ}[-0-/] (3.5,1)
         \Dmor{bcirc}1/2 (4,0)}=
\Diagram{\Dmor{bcirc}[-0-/-] (2,1.5) \Dmor{bcirc}1/2 (3,-0.5)
         \Dmor{circ}[0-0-/-] (1,0.5) \Dmor{circ}[-0-0/-] (3,0.5)
         \draw (0.5,0) -- (0.5,-1); \draw (1.5,0) -- (1.5,-1);}
$$
Having this, proving (2) is already straightforward:
$$
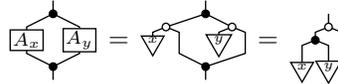
\Diagram{\Dmor{bcirc}[-0-/-] (1,2) \Dmor{bcirc}[-/-0-] (1,-1)
         \DMor{square}1/1 (0,0.5) {$\scriptstyle A_x$}
         \DMor{square}1/1 (2,0.5) {$\scriptstyle A_y$}}=
\Diagram{\Dmor{bcirc}[-0-/-] (2,1.5) \Dmor{bcirc}1/2 (3,-0.5)
         \Dmor{circ}[0-0-/-] (1,0.5) \Dmor{circ}[-0-0/-] (3,0.5)
         \DMor{vec}0/1 (0.5,-1) {$\scriptstyle x$} \DMor{vec}0/1 (1.5,-1) {$\scriptstyle y$}}=
\Diagram{\Dmor{circ}[-0-/-] (2,2) \Dmor{bcirc}[-0-/-] (1,1) \draw (3,1.5) -- (3,-1);
         \DMor{vec}0/1 (0,-0.5) {$\scriptstyle x$} \DMor{vec}0/1 (2,-0.5) {$\scriptstyle y$}}
$$

Finally, (3) and (4) is again rather straightforward:
$$
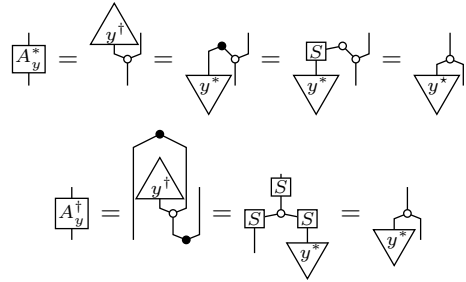
\Diagram{\DMor{square}1/1 (1,0.5) {$\scriptstyle A_x^\dag$}}=
\Diagram{\Dmor{circ}1/2 (1,-.5) \DMor{covec}1/0 (0.5,1) {$\scriptstyle x^\dag$} \draw (1.5,0) -- (1.5,1.5);}=
\Diagram{\Dmor{circ}1/2 (2.5,0.5) \Dmor{bcirc}2/0 (1.5,1.5) \DMor{vec}0/1 (1,0) {$\scriptstyle x^*$}
         \draw (2.5,0) -- (2.5,-1); \draw (3,1) -- (3,1.5);}=
\Diagram{\Dmor{circ}1/2 (2.5,1) \Dmor{circ}2/0 (1.5,2)
         \DMor{square}1/1 (1,0.5) {$\scriptstyle S$}
         \DMor{vec}0/1 (1,-1.5) {$\scriptstyle x^*$}
         \draw (2.5,0.5) -- (2.5,-2); \draw (3,1.5) -- (3,2);}=
\Diagram{\Dmor{circ}2/1 (1.5,1) \DMor{vec}0/1 (1,-0.5) {$\scriptstyle x^\star$} \draw(2,0.5) -- (2,-1);}
$$
$$
\Diagram{\DMor{square}1/1 (1,0.5) {$\scriptstyle A_x^*$}}=
\Diagram{\Dmor{circ}1/2 (1,-.5) \DMor{covec}1/0 (0.5,1) {$\scriptstyle x^\dag$}
         \Dmor{bcirc}[-0-/] (0.5,3) \Dmor{bcirc}0/2 (1.5,-1.5)
         \draw(-0.5,2.5) -- (-0.5,-2); \draw (1.5,0) -- (1.5,2.5); \draw (2,-1) -- (2,3);}=
\Diagram{\DMor{square}1/1 (1,2) {$\scriptstyle S$} \DMor{square}1/1 (0.5,-1) {$\scriptstyle S$} \DMor{square}1/1 (1.5,-1) {$\scriptstyle S$}
         \Dmor{circ}2/1 (1,0.5) \DMor{vec}0/1 (1.5,-3) {$\scriptstyle x^*$} \draw (0.5,-2) -- (0.5,-3);}=
\Diagram{\Dmor{circ}2/1 (1,0.5) \DMor{vec}0/1 (0.5,-1) {$\scriptstyle x^*$} \draw (1.5,0) -- (1.5,-1);}
$$
\end{proof}

\begin{cor}
$A_x$ is a quantum graph if and only if $x$ is a projection in $C(\Gamma)$. In that case, we call it a \emph{Cayley graph}. It has no loops if and only if $x\bullet\epsilon^\dag=0$. It is undirected if and only if $x=x^\star$.
\end{cor}

See also \cite{Was23} for an analogous result formulated in the more general framework of discrete (possibly infinite) quantum groups.

\begin{cor}
Any finite quantum group $\Gamma$ defines a homogeneous quantum coherent algebra $\Alg=\{A_x\mid x\in l^2(\Gamma)\}$, which is isomorphic to $(l^2(\Gamma),\bullet,*,{}^*,{}^\star)$.
\end{cor}

Hence, any finite quantum group $\Gamma$ can be considered as a quantum association scheme generalizing the concept of a group scheme. We will denote it by $\Ass_\Gamma$.

\begin{defn}
Let $\Gamma$ be a finite quantum group. A \emph{quantum Schur ring} is any vector space $S\subset l^2(\Gamma)$, which is closed under all the operations $\bullet,*,{}^*,{}^\star$ and contains both units $\eta$ and $\epsilon^\dag$. That is, it is a unital involutive subalgebra of both $C(\Gamma)$ and $\C\Gamma$. The quantum association scheme $\Ass$ corresponding to a quantum coherent subalgebra $\Alg\subset\spanlin\Ass_\Gamma$ is called a \emph{quantum translation association scheme}.
\end{defn}

By construction, a quantum Schur ring $S$ induces a quantum translation association scheme $\Ass$ by $\spanlin\Ass=\{A_x\mid x\in S\}$ and vice versa. That is, $x\mapsto A_x$ is a functor providing the equivalence between the category of quantum Schur rings and the category of quantum translation association schemes.

Now, recall Proposition \ref{P.FourierOps}, stating that the Fourier transform exchanges the operations $*$ and $\bullet$ and the involutions $\star$ and $*$ in $l^2(\Gamma)$. This obviously restricts to any quantum Schur ring $S\subset l^2(\Gamma)$. By the functor $x\mapsto A_x$ this passes also to the corresponding quantum translation association schemes. To summarize:

\begin{thm}
\label{T}
Let $\Gamma$ be a finite quantum group and $\Ass$ a quantum translation association scheme over $\Gamma$ corresponding to a Schur ring $S\subset l^2(\Gamma)$. Then the translation association scheme $\hat\Ass$ given by
$$\spanlin\hat\Ass=\{\hat A\mid A\in\spanlin\Ass\}=\{A_{\hat x}\mid x\in S\}$$
is dual to $\Ass$.
\end{thm}

If it is clear from context that we consider $\Ass$ as a quantum translation association scheme over a particular quantum group, we can call the above defined $\hat\Ass$ \emph{the} dual of $\Ass$. Nevertheless, note that the quantum group $\Gamma$ may not be defined uniquely; that is, $\Ass$ can be considered as a quantum translation association scheme over various different quantum groups and hence can have various different duals. An example is the complete bipartite graph discussed in Examples \ref{E.Knn},~\ref{E.Knn2}.

Now, let us reformulate the theorem to the particular case when $\Gamma$ is an ordinary non-abelian group.

\begin{prop}
\label{P}
Let $\Gamma$ be a (possibly non-abelian) finite group. Let $\Ass=\{A_i\}_{i=0}^d$ be a commutative translation association scheme over $\Gamma$ and denote by $\{\Gamma_i\}_{i=0}^d$ the corresponding partition of $\Gamma$. For any $g\in\Gamma$, denote by $i(g)$ the index of the corresponding part, so $g\in\Gamma_{i(g)}$. Denote also by $\{E_j\}$ the dual basis of $\Ass$ and by $Q$ the dual eigenvalue matrix of $\Ass$. Then the dual $\hat\Ass$ is a CCQAS given by $\hat\Ass=\{B_j\}_{j=0}^d$ with $B_j=n\hat E_j$ acting by
$$B_j\hat e_g=Q_{i(g)j}\hat e_g.$$
\end{prop}
\begin{proof}
The fact that $\hat\Ass$ is commutative and cocommutative is clear. The only thing we need to prove is that $B_j$ given as above is indeed the basis of Schur idempotents.

By assumptions, we have $A_i=A_{\alpha_i}$ with $\alpha_i=\sum_{g\in\Gamma_i}e_g$. The dual basis is then given by $E_j=\frac{1}{n}\sum_{s=0}^d Q_{sj}A_s$, so we can also write $E_j=A_{\epsilon_j}$ with $\epsilon_j=\frac{1}{n}\sum_{s=0}^d Q_{sj}\alpha_s$ satisfying $\epsilon_j*\epsilon_j=\epsilon_j$ and $\epsilon_j^\star=\epsilon_j$. Now, we can put $B_j=\hat E_j=A_{\hat\epsilon_j}$. Then these satisfy
$$B_j\bullet B_j=A_{\hat\epsilon_j\bullet\hat\epsilon_j}=A_{\widehat{\epsilon_j*\epsilon_j}}=A_{\hat\epsilon_j}=B_j,$$
$$B_j^*=A_{\hat\epsilon_j^*}=A_{\widehat{\epsilon_j^\star}}=A_{\hat\epsilon_j}=B_j.$$

Finally, note that $\hat e_h*\hat e_g=n\,\widehat{e_h\bullet e_g}=n\delta_{gh}\hat e_g$, so $\hat\alpha_i*\hat e_g$ equals to $n\hat e_g$ if $g\in\Gamma_i$ and zero otherwise. This leads to the formula
\[B_j\hat e_g=\hat\epsilon_j*\hat e_g=\frac{1}{n}\sum_{s=0}^d Q_{sj}(\hat\alpha_s*\hat e_g)=Q_{i(g)j}\hat e_g.\qedhere\]
\end{proof}

\begin{rem}
\label{R.Bannai}
As we mentioned in the introduction, some sort of general duality statement for arbitrary association schemes was formulated already in \cite{Ban82}. So, let us rephrase the result of Bannai in our language.

Take arbitrary association scheme $\Ass=\{A_0,\dots,A_d\}$ and consider its Bose--Mesner algebra $\Alg=\spanlin\Ass$. We can consider the dual vector space $\Alg^*$ consisting of all linear functionals $\Alg\to\C$. Since $A_0,\dots,A_d$ is a basis of $\Alg$, we can clearly identify $\Alg^*$ with $C(\Ass)$. By this, $\Alg^*$ gets an algebra structure. Since everything is finite-dimensional, we can further identify $C(\Ass)$ with $\Alg$ as in Example~\ref{E.CGCG}, so that the mentioned algebraic structure in $\Alg^*\simeq C(\Ass)$ corresponds to the Schur product in $\Alg$.

As for the other product (the composition) in $\Alg$, Bannai essentially implicitly observes that it can be encoded in $\Alg^*\simeq C(\Ass)$ via the comultiplication $f\mapsto\Delta(f)$ with $\Delta(f)(a,b)=f(ab)$ for $a,b\in\Alg$ and $f\in\Alg^*$. The main theorem then says that if we take the second dual $\Alg^{**}$, which is obviously linearly isomorphic to $\Alg$, we can reconstruct the composition product via $\phi\cdot\psi=(\phi\otimes\psi)\circ\Delta$ for $\phi,\psi\in\Alg^{**}$.

But this is well known from the theory of coalgebras: The dual of every finite-dimensional algebra has a natural coalgebra structure. Taking the dual second time, we obtain an algebra, which is isomorphic to the original one.

Our result is also based on a well-known fact from the theory of Hopf algebras, namely that the dual of a Hopf $*$-algebra is again a Hopf $*$-algebra (and the second dual is again isomorphic to the original one). But we believe that it is much more interesting as it actually gives a structure of a quantum association scheme to the dual. That is, we start with a quantum association scheme $\Ass$ and we construct a \emph{new} quantum association scheme $\hat\Ass$. This is not the case for the Bannai's result.

On the other hand, the result of Bannai is more general as it works for general association schemes, while our works only for translation (quantum) association schemes. This is natural as our result relies on some underlying Hopf $*$-algebra, i.e.\ a (quantum) group.
\end{rem}

\begin{ex}\label{E.Knn2}
Consider the dihedral group $D_n=\{r^k,sr^k\mid k=0,\dots,n-1\}$, where we denote by $r$ the elementary rotation and by $s$ the reflection. Consider the generating set $S=\{sr^k\mid k=0,\dots,n-1\}$. Then the corresponding Cayley graph is the complete bipartite graph $K_{n,n}$ as in Example~\ref{E.Knn}. This time, the corresponding partition of $\Gamma$ is given by
$$\Gamma_0=\{e\},\quad\Gamma_1=\{sr^k\mid k=0,\dots,n-1\},\quad\Gamma_2=\{r^k\mid k=1,\dots,n-1\}.$$

We can view the corresponding classical association scheme also as a quantum association scheme with $\spanlin\Ass=\spanlin\{A_i\mid i=0,1,2\}$, $A_i=A_{\alpha_i}$ and
$$\alpha_0=e,\qquad\alpha_1=s+sr+\cdots+sr^{n-1},\qquad\alpha_2=r+r^2+\cdots+r^{n-1}.$$
Here, we decided to simplify the notation and write $g$ instead of $e_g$ for the elements of $l^2(D_n)$.

Let us stress that regardless of the way we constructed the association scheme, it literally equals to the association scheme from Example~\ref{E.Knn}. Nevertheless, the construction of the dual will already depend on the underlying group.

In this case, the dual is a quantum association scheme (actually a strongly regular quantum graph) over the quantum space $\hat D_n$ with $C(\hat D_n)\simeq\C D_n$ being non-commutative. Based on the second column of the dual eigenvalue matrix $Q$ given in Example~\ref{E.Knn}, we can express the adjacency matrix $\hat A\colon\C D_n\to\C D_n$ as
\begin{align*}
e&\mapsto (2n-2)e,\\
r^i&\mapsto -2r^i&&\text{for $i=1,\dots,n-1$,}\\
sr^i&\mapsto 0   &&\text{for $i=0,\dots,n-1$}.
\end{align*}

It might be interesting to reveal the explicit structure of this graph. Note that
$$\C D_n=\bigoplus_{\pi\in\Irr D_n}M_{n(\pi)}(\C)=\C\oplus\C({}\oplus\C\oplus\C)\oplus\underbrace{M_2(\C)\oplus\cdots\oplus M_2(\C)}_{\text{$\lfloor(n-1)/2\rfloor$ times}}$$

The group $D_n$ has the following set of inequivalent irreducible representations: It has the trivial representation $\epsilon$, the representation $\tau$ given by $\tau(r)=1$, $\tau(s)=-1$, if $n$ is even, then it also has the representation $\sigma$ given by $\sigma(r)=-1$, $\tau(s)=1$ and the product $\sigma\tau$, and it has the defining two-dimensional representation $\rho$ given by rotations and reflections in the plane and its powers $\rho^k$, $k=1,\dots,\lfloor\frac{n-1}{2}\rfloor$.

Expressing the adjacency matrix $\hat A$ in the basis $\epsilon$, $\tau$, ($\sigma$, $\sigma\tau$,) $\frac{1}{\sqrt2}e_{11}^{(k)}$, $\frac{1}{\sqrt2}e_{12}^{(k)}$, $\frac{1}{\sqrt2}e_{21}^{(k)}$, $\frac{1}{\sqrt2}e_{22}^{(k)}$, $k=1,\dots,\lfloor\frac{n-1}{2}\rfloor$, we reveal some kind of quantum cocktail party structure again.

To include some concrete computation, we can do this for $n=8$. In this case, $\hat A$ in this basis looks as follows
$$\hat A=
\begin{pmatrix}
0&0&1&1&\sqrt2&0&0&\sqrt2\\
0&0&1&1&\sqrt2&0&0&\sqrt2\\
1&1&0&0&\sqrt2&0&0&\sqrt2\\
1&1&0&0&\sqrt2&0&0&\sqrt2\\
\sqrt2&0&0&\sqrt2&1&0&0&1\\
\sqrt2&0&0&\sqrt2&0&-1&1&0\\
\sqrt2&0&0&\sqrt2&0&1&-1&0\\
\sqrt2&0&0&\sqrt2&1&0&0&1\\
\end{pmatrix}
$$
Note that the right-bottom corner corresponds to the \emph{anticommutative square graph} from \cite[Example~1.20]{GQGraph}. Recall in particular that the anticommutative square is quantum isomorphic to the classical square graph, which can be also viewed as a cocktail party graph $K_{2,2}$. 

For general $n$, one can obtain a similar structure that can be described as follows. The quantum graph consists of $\lfloor\frac{n-1}{2}\rfloor$ copies of the anticommutative square and either the classical $K_{2,2}$ (if $n$ is even) or two classical unconnected vertices (if $n$ is odd). Otherwise everything is connected to everything. In other words, it is the complement of the quantum graph given by one or two copies of the classical $K_2$ and $\lfloor\frac{n-1}{2}\rfloor$ copies of the unique quantum graph on $M_2(\C)$ with the ``number of directed edges'' $\eta^\dag A\eta=4$ (which is actually quantum isomorphic to $K_2\sqcup K_2$, see again \cite{GQGraph}).
\end{ex}

\subsection{Quantum Latin squares}
\label{secc.Latin}

We finish the article by introducing another way of constructing strongly regular quantum graphs -- we generalize the concept of Latin squares. Let us first recall the classical setting.

A Latin square is an $n\times n$ matrix $L$ with entries in some $n$-element set $X$ such that each row and each column is some permutation of $X$. An example is the Cayley table of any group $\Gamma$. Indeed, if we define $L_{gh}=gh$, then each row and each column contains all elements of the group exactly once. In general, Latin squares are in bijection with finite quasigroups (a quasigroup is a set $X$ equipped with a possibly non-associative multiplication with unique division).

For any Latin square $L$, we can define a strongly regular graph on $n^2$ vertices corresponding to the entries of the Latin square. Two vertices $(x_1,y_1)$ and $(x_2,y_2)$ in such a graph are connected if and only if $x_1=x_2$ or $y_1=y_2$ or $L_{x_1y_1}=L_{x_2y_2}$. The parameters of such a strongly regular graph are given as $(n^2,3(n-1),n,6)$.

In particular, taking the Cayley table as the Latin square, the associated strongly regular graph is a Cayley graph of the group $\Gamma^{\rm op}\times\Gamma$ with respect to the generating set $S=\{(e,g),(g,e),(g^{-1},g)\mid g\in\Gamma\}$. Here, we denote by $\Gamma^{\rm op}$ the group with the same underlying set as $\Gamma$, but with opposite multiplication $g\cdot^{\rm op}h=hg$. Note that we could also use the isomorphism $\Gamma^{\rm op}\to\Gamma$, $g\mapsto g^{-1}$ and observe that the graph is also a Cayley graph of $\Gamma\times\Gamma$ with respect to $S=\{(e,g),(g,e),(g,g)\mid g\in\Gamma\}$.

Now, we are going to quantize this concept.

\begin{defn}
Let $X$ and $Y$ be finite quantum spaces. We define a new finite quantum space $X\times Y$ by $C(X\times Y)=C(X)\otimes C(Y)$. Here, $\otimes$ denotes the ordinary tensor product of C*-algebras, where the multiplication and involution is taken entrywise, so $(x_1\otimes y_1)(x_2\otimes y_2)=x_1x_2\otimes y_1y_2$ and $(x\otimes y)^*=x^*\otimes y^*$.
\end{defn}

\begin{rem}
There is an important subtlety to observe in the above definition. In the preliminary section on finite quantum spaces, we studied for any finite quanatum space $X$ the tensor powers $l^2(X)^{\otimes k}$. Here, we cosidered the bilinear form $\beta_k$, which induced the notions of categorical transposition $\Ctrans$ and conjugation $*$ for tensors $A\colon l^2(X)^{\otimes k}\to l^2(X)^{\otimes l}$. In particular, for $x,y\in l^2(X)$, we had $(x\otimes y)\Ctrans=y\Ctrans\otimes x\Ctrans$ and $(x\otimes y)^*=y^*\otimes x^*$.

However, considering the space $X\times X$, we have $l^2(X\times X)\simeq l^2(X)\otimes l^2(X)$ as vector spaces, but all the operations are defined entrywise. In particular, taking $x,y\in l^2(X)$ and considering $x\otimes y$ as an element of $l^2(X\times X)$, we have $(x\otimes y)\Ctrans=x\Ctrans\otimes y\Ctrans$ and $(x\otimes y)^*=x^*\otimes y^*$.
\end{rem}

\begin{defn}\label{D.qLatin}
Let $X$ be a finite quantum space. A quantum Latin square on $X$ is a linear map $L\colon l^2(X)\otimes l^2(X)\to l^2(X)$ such that the following holds.
\begin{enumerate}
\item
$
\Diagram{\DMor{square}1/2 (1,0.5) {$\scriptstyle L^\dag$}}
=
\Diagram{\DMor{square}2/1 (2,0.5) {$\scriptstyle L$}
         \Dmor{bcirc}2/0 (1.5,2) \Dmor{bcirc}[/-0-] (2.5,-1) \Dmor{bcirc}[/-0-] (3.5,-1)
         \draw (1,1.5) -- (1,-1); \draw (3.5,-.5) -- (3.5,2); \draw (4.5,-.5) -- (4.5,2);}
$.
\item
$
\Diagram{\DMor{square}2/1 (1,0) {$\scriptstyle L$}
         \Dmor{bcirc}1/2 (1,1.5)}
=
\Diagram{\Dmor{bcirc}[-/-0] (1,-0.5) \Dmor{bcirc}[-/0-] (2,-0.5)
         \draw (1,-0.5) -- (2,1); \draw (2,-0.5) -- (1,1);
         \DMor{square}[-0/-] (1,1) {$\scriptstyle L$} \DMor{square}[0-/-] (2,1) {$\scriptstyle L$}}
$,\quad
$
\Diagram{\DMor{square}2/1 (1,0.5) {$\scriptstyle L$} \Dmor{bcirc}0/0 (1,1.5)}
=
\Diagram{\Dmor{bcirc}1/0 (1,0) \Dmor{bcirc}1/0 (2,0)}
$
.
\item
$
\Diagram{\DMor{square}2/1 (1,0.5) {$\scriptstyle L$} \Dmor{bcirc}0/0 (0.5,-0.5)}
=
\Diagram{\Dmor{bcirc}1/0 (1,0) \Dmor{bcirc}0/1 (1,1)}
$,\quad
$
\Diagram{\DMor{square}2/1 (1,0.5) {$\scriptstyle L$} \Dmor{bcirc}0/0 (1.5,-0.5)}
=
\Diagram{\Dmor{bcirc}1/0 (1,0) \Dmor{bcirc}0/1 (1,1)}$
\end{enumerate}
\end{defn}

To explain the definition, note the following: Axiom (1) means that $L=L^*$ taken as a map $l^2(X\times X)\to l^2(X)$. Axiom (2) equivalently says that $L^\dag$ is a unital homomorphism $C(X)\otimes C(X)\to C(X)$. Together with axiom (1) it means that it is a unital $*$-homomorphism. So, we have something like a finite quantum group (a finite-dimensional Hopf $*$-algebra) with $L$ playing the role of comultiplication except that there is no counit and the comultiplication need not be coassociative. Finally, the third condition is supposed to mean that $L$ as a multiplication admits a unique division. The intuition here could be that $\eta$ is like a ``superposition of all possible inputs''. If we fix one of the inputs $x $and put $\eta$ for the other, we should get ``all possible outputs'' again regardless of $x$. We make this more precise in the following statement:

\begin{prop}
\label{P.Galois}
Let $X$ be a finite quantum space and consider a linear map $L\colon l^2(X)\otimes l^2(X)\to l^2(X)$ satisfying axioms (1) and (2) of Definition~\ref{D.qLatin}. Denote
$$L_r=(L\otimes\id)(\id\otimes m^\dag)=
\Diagram{\DMor{square}2/1 (1,1) {$\scriptstyle L$} \Dmor{bcirc}1/2 (2,-0.5) \draw (0.5,0) -- (0.5,-1); \draw (2.5,0) -- (2.5,2);},
\qquad
  L_l=(\id\otimes L)(m^\dag\otimes\id)=
\Diagram{\DMor{square}2/1 (2,1) {$\scriptstyle L$} \Dmor{bcirc}1/2 (1,-0.5) \draw (2.5,0) -- (2.5,-1); \draw (0.5,0) -- (0.5,2);}.
$$
Then axiom (3) of Definition~\ref{D.qLatin} is equivalent to saying that $L_r$ and $L_l$ are unitary. That is,
$$
\Diagram{\Dmor{bcirc}1/2 (1.5,0.75) \Dmor{bcirc}2/1 (1.5,0.25)
         \DMor{square}[-0/-] (1,1.5) {$\scriptstyle L$} \DMor{square}[-/-0] (1,-.5) {$\scriptstyle L^\dag$}
         \draw (2,1.25) -- (2,2.5); \draw (2,-.25) -- (2,-1.5);}
=\Diagram{\draw (1,-1.5) -- (1,2.5); \draw (2,-1.5) -- (2,2.5);}=
\Diagram{\Dmor{bcirc}1/2 (1.5,-1) \Dmor{bcirc}2/1 (1.5,2)
         \DMor{square}[-0/-] (1,-0.25) {$\scriptstyle L$} \DMor{square}[-/-0] (1,1.25) {$\scriptstyle L^\dag$}
         \draw (2,-0.5) -- (2,1.5); \draw (0.5,2.25) -- (0.5,2.5); \draw (0.5,-1.25) -- (0.5,-1.5);}
,\qquad
\Diagram{\Dmor{bcirc}1/2 (1.5,0.75) \Dmor{bcirc}2/1 (1.5,0.25)
         \DMor{square}[0-/-] (2,1.5) {$\scriptstyle L$} \DMor{square}[-/0-] (2,-.5) {$\scriptstyle L^\dag$}
         \draw (1,1.25) -- (1,2.5); \draw (1,-.25) -- (1,-1.5);}
=\Diagram{\draw (1,-1.5) -- (1,2.5); \draw (2,-1.5) -- (2,2.5);}=
\Diagram{\Dmor{bcirc}1/2 (1.5,-1) \Dmor{bcirc}2/1 (1.5,2)
         \DMor{square}[0-/-] (2,-0.25) {$\scriptstyle L$} \DMor{square}[-/0-] (2,1.25) {$\scriptstyle L^\dag$}
         \draw (1,-0.5) -- (1,1.5); \draw (2.5,2.25) -- (2.5,2.5); \draw (2.5,-1.25) -- (2.5,-1.5);}.
$$
\end{prop}
\begin{proof}
First, suppose that $L$ satisfies Def.~\ref{D.qLatin}. We will prove the unitarity just for $L_r$, the proof for $L_l$ is the same. Since $l^2(X)$ is finite-dimensional, it is enough to prove the first equality.

Composing the first equality in axiom (2) with $\eta\otimes I$, we obtain
$$\Diagram{\Dmor{bcirc}0/2 (1,1) \Dmor{bcirc}1/0 (1,0)}=
\Diagram{\Dmor{bcirc}[/-0-] (1.5,-.5) \Dmor{bcirc}[-/-0-] (2.5,-.5)
         \DMor{square}2/1 (1,1) {$\scriptstyle L$} \DMor{square}2/1 (3,1) {$\scriptstyle L$}}$$
So,
$$
\Diagram{\Dmor{bcirc}1/2 (1.5,0.75) \Dmor{bcirc}2/1 (1.5,0.25)
         \DMor{square}[-0/-] (1,1.5) {$\scriptstyle L$} \DMor{square}[-/-0] (1,-.5) {$\scriptstyle L^\dag$}
         \draw (2,1.25) -- (2,2.5); \draw (2,-.25) -- (2,-1.5);}=
\Diagram{\Dmor{bcirc}[/-0-] (1.5,-.5) \Dmor{bcirc}[0/-0-] (2.5,-.5) \Dmor{bcirc}2/0 (3.5,2.5)
         \DMor{square}2/1 (1,1) {$\scriptstyle L$} \DMor{square}2/1 (3,1) {$\scriptstyle L$}
         \draw (4,2) -- (4,-1.5); \draw (5,2.5) -- (5,-1.5); \Dmor{bcirc}[/-00000] (5,-1)}
=\Diagram{\draw (1,-1.5) -- (1,2.5); \draw (2,-1.5) -- (2,2.5);}.$$

For the converse, observe that by $L^\dag$ being unital (axiom (2)), we have
$$L_r^\dag(\eta\otimes x)=
\Diagram{\DMor{square}1/2 (1,0.5) {$\scriptstyle L^\dag$} \Dmor{bcirc}2/1 (2,2) \Dmor{bcirc}0/1 (1,-1) \DMor{vec}0/1 (2.5,0.5) {$\scriptstyle x$}}=
\Diagram{\DMor{bcirc}0/1 (1,0) {} \DMor{vec}0/1 (2,0) {$\scriptstyle x$}}=
\eta\otimes x.$$
By unitarity of $L_r$, this means that $L_r(\eta\otimes x)=\eta\otimes x$, so $L_r(\eta\otimes I)=\eta\otimes I$. Finally, composing from left with $I\otimes\eta^\dag$, we get exactly
$$
\Diagram{\Dmor{bcirc}1/0 (1,0) \Dmor{bcirc}0/1 (1,1)}=
\Diagram{\DMor{square}2/1 (1,1) {$\scriptstyle L$} \Dmor{bcirc}1/2 (2,-0.5) \Dmor{bcirc}1/0 (2.5,2) \Dmor{bcirc}0/1 (0.5,-0.5) \draw (2.5,0) -- (2.5,1.5);}=
\Diagram{\DMor{square}2/1 (1,0.5) {$\scriptstyle L$} \Dmor{bcirc}0/0 (0.5,-0.5)}
$$
\end{proof}

Let $X$ be an ordinary finite set equipped with some multiplication $\cdot\colon X\times X\to X$, which can be extended to $L\colon l^2(X)\otimes l^2(X)\to l^2(X)$. Then the $L_r$ and $L_l$ stand for the mappings $(g,h)\mapsto (g\cdot h,h)$, resp. $(g,h)\mapsto (g,g\cdot h)$ called sometimes \emph{Galois maps}. They being bijective mean that there is a unique division, so the operation defines a quasigroup.

Thus, to sum up, the pair $(X,L)$ essentially defines a \emph{finite quantum quasigroup}. There are several approaches to define quantum quasigroups in the literature already. The most general one is provided in \cite{Smi16}, where the author does not assume associativity for neither of the two products. Our definition is most similar to \emph{Hopf quasigroups} as defined by Klim and Majid \cite{KM10}, see also \cite[Theorem 2.5]{Brz10}. There are, however, a couple of differences. We do not require the existence of a unit for the multiplication $L$. On the other hand, we require the $*$ structure and the Frobenius structure for the algebra $C(X)$.

\begin{ex}
Quantum Latin squares over the classical space $X=\{1,\dots,n\}$ are exactly the classical $n\times n$ Latin squares. The correspondence goes by Gelfand duality and is proven in a similar way as the correspondence between groups and Hopf $*$-algebras assuming that the multiplication $m$ is commutative.
\end{ex}

\begin{ex}
Any finite quantum group $\Gamma$ is a quantum quasigroup and hence defines a quantum Latin square over $\Gamma$ by $L=\Delta^\dag$.
\end{ex}

\begin{rem}
The notion of quantum Latin squares was already introduced by Musto and Vicary in \cite{MV15}. But their definition is different and not compatible with ours. The point is that they are quantizing Latin squares in a different way: They keep the classical space of $n$ points $X=\{1,\dots,n\}$ and make the entries of the Latin square quantum. Consequently, their quantum Latin squares need not satisfy axioms (2) and (3). In contrast, we allow quantum spaces, but keep the strict axioms. This is similar to the situation with Hadamard matrices, where we also have two competing quantizations as discussed in \cite[Remark~2.6]{GroQHad}.

It would be possible to generalize both definitions and introduce \emph{quantum quantum Latin squares} by Definition~\ref{D.qLatin} excluding axioms (2) and (3) (replacing them with the assumption that the Galois maps are unitary). We will not do that here as the resulting structure would not yield any strongly regular quantum graphs.
\end{rem}

\begin{lem}\label{L.Latinsrg}
Let $L$ be a quantum Latin square over a finite quantum space $X$ with $n:=\eta^\dag\eta$. Then the following hold.
\begin{enumerate}
\item $(L^\dag L)^*=L^\dag L$ taken as a linear map $l^2(X\times X)\to l^2(X\times X)$.
\item $(L^\dag L)\bullet (L^\dag L)=L^\dag L$
\item $(L^\dag L)\bullet (I\otimes J)=I\otimes I=(I\otimes J)(L^\dag L)$.
\item $(L^\dag L)\bullet (J\otimes I)=I\otimes I=(J\otimes I)(L^\dag L)$.
\item $(L^\dag L)\bullet (I\otimes I)=I\otimes I=(I\otimes I)(L^\dag L)$.
\item $(L^\dag L)(I\otimes J)=(J\otimes J)=(I\otimes J)(L^\dag L)$.
\item $(L^\dag L)(J\otimes I)=(J\otimes J)=(J\otimes I)(L^\dag L)$.
\item $LL^\dag=nI$.
\end{enumerate}
\end{lem}
\begin{proof}
\begin{enumerate}
\item This follows directly from axiom (1): $L=L^*$.
\item This can be easily computed from axiom (2).
\item The second equality $(I\otimes J)(L^\dag L)=I\otimes I$ follows from the unitarity $L_l^\dag L_l=I\otimes I$ formulated in Prop.~\ref{P.Galois}. We get the other one by applying the $*$.
\item The same as above: The first equality follows from $L_r^\dag L_r=I$ while the second is obtained by applying the $*$.
\item This is just a consequence of the two items above as $J$ is the identity with respect to the Schur product, so $I\otimes I=(I\otimes J)\bullet(J\otimes I)=(J\otimes I)\bullet(I\otimes J)$.
\item This follows by unitality of $L^\dag$ (axiom (2)).
\item The same as above.
\item This follows from $L_rL_r^\dag=I\otimes I$ by sandwiching with $\eta$ and $\eta^\dag$ in the second tensor factor.
\end{enumerate}
\end{proof}

Note that from (1) and (2), it follows that $L^\dag L$ defines a quantum graph. But, we are going to study a slightly different one:

\begin{prop}\label{P.Latinsrg}
Let $L$ be a quantum Latin square over a finite quantum space $X$. Then the following adjacency matrix defines a strongly regular quantum graph over the quantum space $X\times X$ with parameters $(n^2,3(n-1),n,6)$:
\begin{equation}
\label{eq.LatinSRG}
A=I\otimes J+J\otimes I+L^\dag L-3I\otimes I.
\end{equation}
\end{prop}
\begin{proof}
We need to check that $A\bullet A=A$, $A^*=A$, and $A^2=3(n-1)I\otimes I+n A+6(J\otimes J-A-I\otimes I)$. All this is straightforward to do using Lemma~\ref{L.Latinsrg}.
\end{proof}

It is worth having look at the eigenvalues of such a strongly regular graph. Using Formula~\eqref{eq.SRGeigenmatrix} and substituting the above given parameters, we get
\begin{equation}\label{eq.LatinP}
P=\begin{pmatrix}
1&3(n-1)&(n-1)(n-2)\\
1&n-3&-(n-2)\\
1&-3&2
\end{pmatrix}
\end{equation}
In addition, a straightforward computation gives $P^2=n^2I$, so $Q=P$. This means that all such strongly regular quantum graphs corresponding to quantum Latin squares are self-dual. Hence, for a fixed size of the Latin square $n$, they are also all pairwise dual to each other.

Finally, we are going to show that if the quantum Latin square comes from a finite quantum group, then the associated strongly regular quantum graph is Cayley. Given a quantum group $\Gamma$, we denote by $\Gamma^{\rm op}$ the quantum group with the opposite comultiplication (that is, the associated Hopf $*$-algebra is coopposite, not opposite). By opposite compultiplication, we mean $\Delta^{\rm op}=\Sigma\Delta$, where $\Sigma\colon x\otimes y\mapsto y\otimes x$. We will say that $\Gamma$ is \emph{abelian} if $\Delta^{\rm op}=\Delta$ (that is, if the associated Hopf algebra is cocommutative).

\begin{prop}
Let $\Gamma$ be a finite quantum group. Then the strongly regular graph corresponding to the quantum Latin square $L=\Delta^\dag$ is a Cayley quantum graph corresponding to the quantum group $\Gamma^{\rm op}\times\Gamma$ and the element
$$x=\epsilon^\dag\otimes\eta+\eta\otimes\epsilon^\dag+\epsilon\Delta-3\epsilon^\dag\otimes\epsilon^\dag\in l^2(\Gamma^{\rm op}\times\Gamma).$$
\end{prop}
\begin{proof}
We need to show that $A_x$ with $x$ given as above equals to Eq.~\eqref{eq.LatinSRG} with $L=\Delta$. We can do that term by term, most being immediatelly clear as $A_{\epsilon^\dag}=I$ and $A_\eta=J$ in both $l^2(\Gamma)$ and $l^2(\Gamma^{\rm op})$. It remains to prove that $A_{\epsilon\Delta}=\Delta\Delta^\dag$ in $l^2(\Gamma^{\rm op}\otimes\Gamma)$.

We can show that using the Frobenius law:
$$
\Diagram{\Dmor{circ}2/1 (1,1) \Dmor{circ}2/1 (3,1)
         \Dmor{circ}[/-0-] (1.5,-0.5) \draw (2.5,0.5) -- (2.5,0); \draw (3.5,0.5) -- (3.5,-.5);
         \draw (1.5,0.5) .. controls (1.5,0) and (0.5,0.5) .. (0.5,0);
         \draw (0.5,0.5) .. controls (0.5,0) and (2.5,0) .. (2.5,-.5);}
=
\Diagram{\Dmor{circ}2/1 (1,1) \Dmor{circ}2/1 (3,1) \Dmor{circ}0/2 (2,0)
         \draw (0.5,0.5) -- (0.5,-0.5); \draw (3.5,0.5) -- (3.5,-0.5);}
=
\Diagram{\Dmor{circ}2/1 (1,1) \Dmor{circ}1/2 (2,0)
         \draw (0.5,0.5) -- (0.5,-0.5); \draw (2.5,0.5) -- (2.5,1.5);}
=
\Diagram{\Dmor{circ}1/2 (1,1) \Dmor{circ}2/1 (1,0)}
$$
\end{proof}

\begin{ex}
Let $\Gamma$ be a finite (possibly non-commutative) group. Then its dual $\hat\Gamma$ is an abelian quantum group, so $\hat\Gamma^{\rm op}=\hat\Gamma$. If we denote $l^2(\hat\Gamma)=\{\hat e_g\mid g\in\Gamma\}$, then the comultiplication on $\hat\Gamma$ maps $\hat e_g\mapsto \hat e_g\otimes\hat e_g$ and the unit is given by $\hat e_e$, where the index $e$ stands for the group identity. Consequently, we get a strongly regular quantum graph on the quantum space $\hat\Gamma\times\hat\Gamma$ with adjacency matrix
$$
A\colon \hat e_g\otimes\hat e_h\mapsto (n\delta_{ge}+n\delta_{he}+n\delta_{gh}-3)\hat e_g\otimes\hat e_h,
$$
where $n=|\Gamma|$.

That is, the adjacency matrix is diagonal in the basis $(\hat e_g\otimes\hat e_h)$. The corresponding eigenspaces are $V_0=\{(\hat e_e\otimes\hat e_e)\}$ (i.e. multiplicity 1) with eigenvalue $3(n-1)$, $V_1=\{(\hat e_e\otimes\hat e_h),(\hat e_g\otimes\hat e_e),(\hat e_g\otimes\hat e_g)\mid g\neq e\}$ (multiplicity $3(n-1)$) with eigenvalue $n-3$ and, finally, $V_2=\{(e_g\otimes e_h)\mid g,h\neq e\}$ (multiplicity $(n-1)(n-2)$) with eigenvalue $-3$. This exactly corresponds to the general computation of from Equation~\eqref{eq.LatinP} (eigenvalues of $A$ should be the second column of $P$; the multiplicities should be the first row of $Q=P$).

The strongly regular quantum graph we constructed here is actually \emph{the} dual (in the sense of Section~\ref{secc.trans}) of the classical strongly regular graph corresponding to the classical Latin square of $\Gamma$. Indeed, the classical graph is the Cayley graph of $\Gamma\times\Gamma$ corresponding to the generating set $S=\{(e,g),(g,e),(g,g)\mid g\in\Gamma\}$. The dual eigenvalues coincide with the eigenvalues and are again given by the matrix $Q=P$ from Equation~\eqref{eq.LatinP}. Consequently, by Proposition~\ref{P}, the dual association scheme is exactly given as above.
\end{ex}

\bibliographystyle{halpha}
\bibliography{mybase}

\end{document}